\documentclass[times,sort&compress,3p]{elsarticle}
\journal{}

\RequirePackage{amsthm,amsmath,amsfonts,amssymb}
\RequirePackage[colorlinks,citecolor=blue,urlcolor=blue]{hyperref}
\RequirePackage{graphicx}
\RequirePackage{bm}
\RequirePackage{color}
\RequirePackage{calrsfs}
\RequirePackage{float}
\RequirePackage{verbatim}
\RequirePackage[linesnumbered,ruled]{algorithm2e}
\RequirePackage{accents}
\RequirePackage{enumitem}
\setitemize{noitemsep,topsep=0pt,parsep=0pt,partopsep=0pt}
\RequirePackage{xr}
\usepackage{pdfpages}

\makeatletter
\def\ps@pprintTitle{%
 \let\@oddhead\@empty
 \let\@evenhead\@empty
 \def\@oddfoot{}%
 \let\@evenfoot\@oddfoot}
\makeatother

\usepackage[labelfont=bf]{caption}

\theoremstyle{plain}
\newtheorem{prop}{Proposition}
\newtheorem{defn}[prop]{Definition}
\newtheorem{thm}[prop]{Theorem}
\newtheorem{cor}[prop]{Corollary}
\newtheorem{lem}[prop]{Lemma}
\newtheorem{cond}[prop]{Condition}

\newtheoremstyle{remark}
  {}{}{}{}{\bfseries}{.}{.5em}{{\thmname{#1 }}{\thmnumber{#2}}{\thmnote{ (#3)}}}
\theoremstyle{remark}
\newtheorem{remark}[prop]{Remark}
\newtheorem{example}[prop]{Example}

\newcommand{\change}[1]{{\color{black}{#1}}}

\newcommand{\eps}{\varepsilon}
\newcommand{\N}{\mathbb{N}}
\newcommand{\R}{\mathbb{R}}
\newcommand{\Ac}{\mathcal{A}}
\newcommand{\Bc}{\mathcal{B}}
\newcommand{\Cc}{\mathcal{C}}
\newcommand{\Dc}{\mathcal{D}}
\newcommand{\Ec}{\mathcal{E}}
\newcommand{\Fc}{\mathcal{F}}

\newcommand{\Pc}{\mathcal{P}}
\newcommand{\Mc}{\mathcal{M}}
\newcommand{\Xc}{\mathcal{X}}
\newcommand{\Jc}{\mathcal{J}}
\newcommand{\Kc}{\mathcal{K}}
\newcommand{\Lc}{\mathcal{L}}
\newcommand{\Err}{\mathrm{Err}}

\newcommand{\dd}{\mathrm{d}}
\newcommand{\1}{\mathbf{1}}
\newcommand{\un}[1]{\underaccent{\bar}{#1}}
\newcommand{\supp}{\operatorname{supp}}

\allowdisplaybreaks

\begin{document}

\begin{frontmatter}

\title{An iterated $I$-projection procedure for solving the generalized minimum information checkerboard copula problem}

\author[A1]{Ivan Kojadinovic\corref{mycorrespondingauthor}}
\author[A2]{Tommaso Martini}

\address[A1]{CNRS / Universit\'e de Pau et des Pays de l'Adour / E2S UPPA, Laboratoire de math\'ematiques et applications -- IPRA, UMR 5142, B.P. 1155, 64013 Pau Cedex, France.}
\address[A2]{Dipartimento Interateneo di Scienze, Progetto e Politiche del Territorio, Universita degli Studi di Torino, Italy.}

\cortext[mycorrespondingauthor]{Corresponding author. Email address: \url{ivan.kojadinovic@univ-pau.fr}}

\begin{abstract}
   The minimum information copula principle initially suggested in \cite{MeeBed97} is a maximum entropy-like approach for finding the least informative copula, if it exists, that satisfies a certain number of expectation constraints specified either from domain knowledge or the available data. We first propose a generalization of this principle allowing to change the reference copula and the inclusion of additional constraints fixing certain higher-order margins of the copula. We next show that the associated optimization problem has a unique solution under a natural condition. As the latter problem is intractable in general we consider its version with all the probability measures involved in its formulation replaced by checkerboard approximations. This amounts to attempting to solve a so-called discrete $I$-projection linear problem. We then exploit the seminal results of \citet{Csi75} to derive an iterated procedure for solving the latter and provide theoretical guarantees for its convergence. The usefulness of the procedure is finally illustrated via numerical experiments in dimensions up to four with substantially finer discretizations than those encountered in the literature.
\end{abstract}

\begin{keyword}
checkerboard copulas \sep
dependence modeling \sep
Kullback--Leibler divergence minimization \sep
generalized iterative scaling \sep
marginal compatibility problem \sep
maximum entropy principle \sep
information projections.
\MSC[2020] Primary 60E05, 62B11.
Secondary 65K10.
\end{keyword}

\end{frontmatter}


\section{Introduction}

Let $d \geq 2$ and let $\bm X = (X_1, \dots, X_d)$ be a random vector of interest whose $d$-dimensional distribution function (d.f.) $F$ is assumed to be continuous. In many fields such as environmental modeling \citep{SalDeMKotRos07}, quantitative risk management \citep{McNFreEmb15} or econometric modeling \citep{Pat12}, one wishes to model $F$. Quite often, a practitioner will have some idea about how to model the univariate margins $F_1, \dots, F_d$ of $F$. According to the celebrated work of \citet{Skl59}, to complete the modeling of $F$, one then simply needs to model the unique copula $C$ -- merely the restriction to $[0,1]^d$ of a $d$-dimensional d.f.\ with standard uniform margins --  arising in the following well-known representation of $F$:
\begin{equation*}
\label{eq:F}
F(\bm x) = C\left(F_1(x_1),\dots,F_d(x_d)\right), \qquad \bm x \in \R^d.
\end{equation*}
The issue of estimating $C$ from available realizations of $\bm X$ has been extensively addressed in the literature \cite[see, e.g.,][Chapter~4 and the references therein]{HofKojMaeYan18}. This work is concerned with situations in which it is impossible or difficult to carry out a statistical modeling of $C$. Impossibility arises for instance when $X_1,\dots,X_d$ have not all been observed simultaneously, so that no realizations of $\bm X = (X_1, \dots, X_d)$ are available.

\begin{example}
  \label{ex:risk}
  We illustrate the latter situation by considering the field of quantitative risk management. Therein, $X_1,\dots,X_d$ could for instance represent $d$ types of losses of a financial institution. Since different type of losses \change{cannot necessarily be observed simultaneously}, it may happen in applications that no realizations of $\bm X$ are actually available. One would however typically have realizations from each component of $\bm X$ so that the marginal univariate distributions of $\bm X$ could be estimated. It is also possible that simultaneous realizations of, say, certain pairs or triplets of components of $\bm X$ have been observed so that the corresponding low-dimensional distributions may be estimated.  \qed
\end{example}

When $d=2$ and it is not possible to carry out a ``classical'' statistical modeling of~$C$, \citet{MeeBed97} and \citet{BedWil14} suggested to determine~$C$ via a maximum entropy-like approach \citep{Jay57} which they called the minimum information copula principle. Informally, the idea is to find the least informative bivariate copula, if it exists, that satisfies a certain number of expectation constraints specified either from domain knowledge or the available limited data. A prototypical constraint consists of fixing the value of Spearman's rho \cite{MeeBed97}.

Let us formulate a $d$-dimensional version of the minimum information copula principle. Let $\Mc(\R^d)$ denote the set of probability measures on the Borel sets $\Bc_{\R^d}$ of $\R^d$. With the convention that $0 \log 0 := 0$, for any $P, Q \in \Mc(\R^d)$, let
\begin{equation}
  \label{eq:KL}
  I(P \| Q) :=
  \begin{cases}
    \displaystyle \int_{\R^d} p_Q \log p_Q \, \dd Q, \qquad & \text{if } P \ll Q, \\
    \infty, & \text{otherwise},
  \end{cases}
\end{equation}
where $P \ll Q$ means that $P$ is absolutely continuous with respect to $Q$ and $p_Q$ is the Radon--Nikodym derivative of $P$ with respect to $Q$, i.e., $p_Q = \dd P / \dd Q$. The quantity $I(P \| Q)$ is classically known as the Kullback--Leibler divergence, the information divergence or the relative entropy of $P$ with respect to $Q$.  Note that $I(P \| Q)$ can be equal to $\infty$ even when $P \ll Q$ \citep[see, e.g.,][Section~2.1]{PolWu24}. However, when the support of $Q$ is finite, $I(P \| Q) < \infty$ if and only if $P \ll Q$.

As we continue, for any strictly positive integer~$s$, $[s] := \{1,\dots,s\}$. Furthermore, for any $P \in \Mc(\R^d)$, let $P^{(\{1\})}, \dots, P^{(\{d\})}$  be its univariate margins, that is, the probability measures in $\Mc(\R)$ defined by
$$
P^{(\{\ell\})}(B) := P \left( \left\{ \bm v \in \R^d : v_\ell \in B
  \right\} \right), \qquad B \in \Bc_{\R}, \, \ell \in [d].
$$
Now, let $\Mc([0,1]^d)$ be the subset of $\Mc(\R^d)$ consisting of probability measures whose supports are included in $[0,1]^d$, let $\Cc([0,1]^d)$ be the subset of $\Mc([0,1]^d)$ consisting of probability measures corresponding to $d$-dimensional copulas (such probability measures will also be called $d$-stochastic measures following \citet{LiMikTay98} and \citet[Section 3.1]{DurSem15}) and let $U_d \in \Cc([0,1]^d)$ be the probability measure of the uniform distribution on $[0,1]^d$. Note that the copula of $U_d$ is the so-called $d$-dimensional independence copula and that any $P \in \Cc([0,1]^d)$ satisfies $P^{(\{\ell\})} = U_d^{(\{\ell\})} = U_1$, $\ell \in [d]$, where $U_1$ is the probability measure of the univariate standard uniform distribution. Let $g_1,\dots,g_M$ be $M \geq 1$ continuous functions on $[0,1]^d$ and let $\alpha_1,\dots,\alpha_M \in \R$. With the above notation, the $d$-dimensional version of the minimum information copula problem studied in \cite{BedWil14} is:
\begin{equation}
  \label{eq:MIC}
  \begin{split}
    \min_{P \in \Cc([0,1]^d)} \,& I(P \| U_d) \text{ subject to } \\
    &\int_{[0,1]^d} g_m(\bm v) \dd P(\bm v) = \alpha_m, m \in [M].
  \end{split}
\end{equation}
Roughly speaking, Problem~\eqref{eq:MIC} aims at finding the closest copula to the independence copula (in terms of the Kullback--Leibler divergence) satisfying the $M$ expectation constraints, if it exists. Equivalently, it aims at finding the maximum entropy (that is, the ``least specific'') copula satisfying the $M$ expectation constraints, if it exists. Note in passing that problems similar to~\eqref{eq:MIC} also arise for instance in entropy-regularized optimal transport when additional linear constraints are imposed \cite[see, e.g.,][]{Zae15,EckKup21,LinSte23}. This will be discussed further in forthcoming Remark~\ref{rem:OT}.

Minimum information copula problems of the form~\eqref{eq:MIC} were investigated by several authors. To the best of our knowledge, the only $d$-dimensional studies are \cite{PiaHowBor12,BorHow19}. \citet{PiaHowBor12} considered the situation in which the $M$ expectation constraints in~\eqref{eq:MIC} correspond to fixing the $d(d-1)/2$ Spearman's correlation coefficients of $(X_i,X_j)$, $1 \leq i < j \leq d$, while \citet{BorHow19} extended the previous work to allow mixed moment constraints. \citet{BedWil14} considered arbitrary expectation constraints in a bivariate setting and investigated the form of the solution using results in \cite{Lan73,Nus89,BorLewNus94}. More recently, \citet{SukSei25a} studied an extension of Problem~\eqref{eq:MIC} for $d=2$ in which $M=1$ but the corresponding constraint cannot be interpreted as an expectation anymore, as it consists of fixing Kendall's tau of the minimum information copula. Interestingly enough, \citet{SukSei25} ended up showing that, in the bivariate case, the minimum information copula under fixed Kendall's tau is the Frank copula. Unfortunately, for arbitrary expectation constraints, Problem~\eqref{eq:MIC} is intractable in general. For that reason, \citet{PiaHowBor12}, \citet{BedWil14} and \citet{SukSei25a} all ended up solving simplified versions of problems similar to~\eqref{eq:MIC} using numerical schemes or greedy algorithms. Specifically, they more or less explicitly considered versions of their initial problems with all the probability measures involved in their formulations replaced by so-called checkerboard approximations. The latter can be regarded as applying the aforementioned maximum entropy principle to the class of so-called checkerboard copulas \citep[see, e.g.,][]{LiMikSheTay97,LiMikTay98,CotPfe14} thereby leading to the minimum information checkerboard copula problem.

A first contribution of this work is the proposal of a more general version of the minimum information copula problem in~\eqref{eq:MIC} allowing the inclusion of additional constraints fixing certain higher-order margins of the copula and the possibility of replacing $U_d$ by another application-relevant probability measure (see for instance Example~\ref{ex:costal}). A second contribution is the explicit statement of all the steps leading to its checkerboard version. As we shall see in Section~\ref{sec:MIC}, the resulting generalized minimum information checkerboard copula problem can next be reformulated as a so-called discrete $I$-projection linear problem, where the expression ``$I$-projection'' is used in the sense of the seminal work of \citet{Csi75}. The main contribution of this work is then the proposal of an iterated $I$-projection procedure for solving the latter. The $I$-projections in the proposed procedure consist either of classical marginal scalings as in the well-known iterated proportional fitting procedure \citep[see, e.g.][]{Kru37,DemSte40,Kni08,BroLeu18,mipfp} -- also known as Sinkhorn--Knopp's algorithm for instance in computational optimal transport \cite[see, e.g.,][]{PeyCut19} -- or of specific (approximations of) $I$-projections for each of the expectation constraints, if any. The latter are carried out either via generalized iterative scaling \citep[see, e.g.,][]{DarRat72,Csi89} or using a possibly new result (see Proposition~\ref{prop:gen:h}). Conditions under which the proposed procedure converges are formally established. From a practical perspective, our numerical experiments \change{with $d \in \{2, 3, 4\}$} illustrate that the proposed algorithmic approach can be used to approximately solve checkerboard versions of generalizations of problem~\eqref{eq:MIC} for substantially finer discretizations than for instance those considered in \cite{PiaHowBor12,BorHow19}. 

The outline of this work is as follows. Section~\ref{sec:prelim} formally defines higher-order margins of probability measures, probability and copula arrays, checkerboard approximations of $d$-stochastic measures and $I$-projections. In the third section, we propose a more general version of the minimum information copula problem, show that under a natural condition it has a unique solution, study its checkerboard version and verify that the latter is a particular instance of the well-studied problem which consists of attempting to $I$-project a probability measure with finite support on a so-called linear family of probability measures \citep[see, e.g.,][]{Csi75,CsiShi04}. Section~\ref{sec:solving} then consists of exploiting the seminal results in \cite{Csi75,Csi89} to derive an iterated $I$-projection procedure for solving the studied generalized minimum information checkerboard copula problem. The resulting algorithm is similar in spirit to the so-called RBI-SMART algorithm of \citet{Byr98} used in image processing \citep[see also][Section 3]{LinSte23}. The usefulness of the procedure is illustrated via numerous numerical experiments (some of which are reported in the supplement \cite{KojMar25} and connected to the so-called marginal compatibility problem \citep[see, e.g.,][and the references therein]{Rus91,Joe97,DurKleQue08,FriCha13}). We end this work by mentioning several possible extensions of the considered approach.


\section{Preliminaries}
\label{sec:prelim}

To carry out the promised derivations, we first need to introduce additional notation and definitions. These are related to higher-order margins of probability measures, probability arrays and copula arrays, checkerboard approximations of $d$-stochastic measures and $I$-projections.

\subsection{Higher-order margins of a probability measure}
\label{sec:margins}

For any $J =  \{\ell_1,\dots,\ell_{|J|}\} \subset [d]$, $1 \leq \ell_1 < \dots < \ell_{|J|} \leq d$,  let $\pi_J$ be the function from $\R^d$ to $\R^{|J|}$ defined by
\begin{equation*}
  \pi_J(\bm v) := (v_{\ell_1},\dots,v_{\ell_{|J|}}), \qquad \bm v \in \R^d.
\end{equation*}
In other words, for a vector $\bm v \in \R^d$, the so-called canonical projection $\pi_J$ removes from $\bm v$ its components $v_\ell$ such that $\ell \not \in J$. As we continue, given $\bm v \in \R^d$, we shall also write $\bm v_{J} \in \R^{|J|}$ for $\pi_J(\bm v)$ and $\bm v_{-J} \in \R^{d - |J|}$ for $\pi_{[d] \setminus J}(\bm v)$.

For any $J \subset [d]$, $J \neq \emptyset$, $\pi_J$ allows us to define the $J$-margin of a probability measure $P \in \Mc(\R^d)$ as
\begin{equation}
  \label{eq:P:J}
  P^{(J)} := P \circ \pi_J^{-1}  \in \Mc(\R^{|J|}),
\end{equation}
where
$$
P \circ \pi_J^{-1}(B) = P\big( \{\bm v \in \R^{d}: \bm v_J \in B \} \big), \qquad B \in \Bc_{\R^{|J|}},
$$
with $\Bc_{\R^{|J|}}$ the Borel sets of $\R^{|J|}$. It follows that any probability measure $P \in \Mc(\R^d)$ has $2^d - 2$ ``proper'' margins corresponding to $J \subsetneq [d]$, $J \neq \emptyset$, in~\eqref{eq:P:J}.

\subsection{Probability arrays and copula arrays}
\label{sec:arrays}

Let $n \geq 2$ be a fixed parameter and consider
\begin{equation*}
  [n]^d = \left\{ \bm i = (i_1,\dots,i_d) : i_1 \in [n], \dots, i_d \in [n] \right\}.
\end{equation*}
As we shall keep $n$ fixed in the rest of this work (and thus not attempt asymptotic investigations as $n$ tends to $\infty$ -- see Section~\ref{sec:conclusion} for future work on such aspects), we shall sometimes drop the dependence on $n$ in the forthcoming notation.

Let $\Ac_{d,n}$ be the set of all $d$-dimensional arrays (hypermatrices) whose dimension sizes are all $n$. Any array $a \in \Ac_{d,n}$ can be expressed explicitly in terms of its elements as $(a_{\bm i})_{\bm i \in [n]^d}$. Let $\Pc_{d,n}$ be the subset of $\Ac_{d,n}$ consisting of arrays whose elements are nonnegative and sum up to one. We call the elements of $\Pc_{d,n}$ probability arrays.

Let $\Mc([n]^d)$ be the subset of $\Mc(\R^d)$ consisting of probability measures whose supports are included in $[n]^d$. As we continue, elements of $\Mc([n]^d)$ will be denoted using an underlined capital letter, e.g., $\un{P},\un{Q},\un{R},\dots$ It is easy to verify that probability arrays in $\Pc_{d,n}$ are in one-to-one correspondence with probability measures in $\Mc([n]^d)$: the former can be seen as encoding the values of the probability mass functions (p.m.f.s) of the latter \citep[see, e.g.,][Section~3.1]{GeeKojMar25}. In the sequel, probability arrays in $\Pc_{d,n}$ corresponding to probability measures $\un{P},\un{Q},\un{R},\dots$ in $\Mc([n]^d)$ will always be denoted by the corresponding lowercase letters $p,q,r,\dots$, and vice versa. The support of a $p \in \Pc_{d,n}$ is defined as $\supp(p) = \left\{ \bm i \in [n]^d : p_{\bm i} > 0 \right\}$. Furthermore, for any $p \in \Pc_{d,n}$ and $J \subset [d]$, $J \neq \emptyset$, the $J$-margin $p^{(J)}$ of $p$ is defined as the probability array in $\Pc_{|J|, n}$ corresponding to the $J$-margin $\un{P}^{(J)} \in \Mc([n]^{|J|})$ of $\un{P}$. Alternatively, the array $p^{(J)}$ can be recovered by summing the elements of $p$ along the dimensions not in~$J$. Specifically, we shall express $[n]^{|J|}$ as
$$
[n]^{|J|} = \{ \bm i_J : \bm i \in [n]^d \},
$$
where the notation $\bm i_J$ is defined in Section~\ref{sec:margins} and should be understood as a vector of $|J|$ indices corresponding to the dimensions in $J$. We can then write
$$
p^{(J)}_{\bm i_J} = \sum_{\bm i_{-J} \in [n]^{d-|J|}} p_{\bm i}, \, \bm i_J \in [n]^{|J|}, \qquad \text{and} \qquad p^{(J)} = (p^{(J)}_{\bm i_J})_{\bm i_J \in [n]^{|J|}}.
$$

Finally, following \cite{GeeKojMar25}, we call copula array any $p \in \Pc_{d,n}$ that has uniform univariate margins (that is, $p^{(\{\ell\})}_i = 1/n$, for all $i \in [n]$ and $\ell \in [d]$). The set of all $d$-dimensional copula arrays with dimension sizes all equal to $n$ will be denoted by $\Cc_{d,n}$ as we continue.

\subsection{Checkerboard probability measures and copulas}
\label{sec:approx}

The aim of this section is to explain how simple absolutely continuous approximations of probability measures in $\Cc([0,1]^d)$ can be obtained via a regular partitioning of $[0,1]^d$ and copula arrays in $\Cc_{d,n}$.

Recall that $n \geq 2$ is fixed. We see it now as a discretization parameter. Let $A_1 := [0, 1/n]$, let $A_i := ((i-1)/n, i/n]$, $i \in \{2,\dots,n\}$, and let
\begin{equation}
  \label{eq:B:i}
  B_{\bm i} := A_{i_1} \times \dots \times A_{i_d}, \qquad \bm i \in [n]^d.
\end{equation}
Then $\{B_{\bm i} : \bm i \in [n]^d\}$ is a partition of $[0,1]^d$ into $n^d$ hypercubes, each of volume $n^{-d}$. Next, let $p \in \Cc_{d,n}$. By analogy with the construction initially considered in \cite{LiMikSheTay97} and using the terminology suggested in \cite{CotPfe14}, we call
\begin{equation}
  \label{eq:check:p}
  \check f_p(\bm v) :=
  \begin{cases}
    n^d \sum_{\bm i \in [n]^d} \1_{B_{\bm i}}(\bm v) p_{\bm i}, &\text{if } \bm v \in [0,1]^d, \\
    0, &\text{otherwise}.
  \end{cases}
\end{equation}
the checkerboard density with skeleton $p$. The latter is merely a piecewise constant $d$-dimensional density with value $n^d p_{\bm i}$ on each $B_{\bm i}$ in~\eqref{eq:B:i}, and zero elsewhere. The corresponding probability measure in $\Mc([0,1]^d)$ is
\begin{equation}
  \label{eq:check:P}
  \check P(B) = \int_B \check f_p(\bm v) \dd \bm v = n^d \sum_{\bm i \in [n]^d} p_{\bm i}  \int_{B \cap B_{\bm i}} \dd \bm v, \qquad B \in \Bc_{[0,1]^d}.
\end{equation}
It can be verified that $\check f_p$ in~\eqref{eq:check:p} has standard uniform margins which implies that $\check P$ in~\eqref{eq:check:P} actually belongs to $\Cc([0,1]^d)$. As we continue, the subset of $\Cc([0,1]^d)$ consisting of probability measures of the form~\eqref{eq:check:P} with $p \in \Cc_{d,n}$ will be denoted by $\check \Cc_n([0,1]^d)$ and its elements will always be denoted using an accentuated capital letter, e.g., $\check P, \check Q, \check R,\dots$

Clearly, $\check \Cc_n([0,1]^d)$ is in one-to-one correspondence with~$\Cc_{d,n}$. The latter follows from~\eqref{eq:check:P} and the fact that the skeleton $p \in \Cc_{d,n}$ in the latter expression can be recovered from $\check P$ via $p_{\bm i} = \check P(B_{\bm i})$, $\bm i \in [n]^d$, where $B_{\bm i}$ is defined in~\eqref{eq:B:i}. In the rest of this work, copula arrays in $\Cc_{d,n}$ corresponding to probability measures $\check P, \check Q, \check R,\dots$ in $\check \Cc_n([0,1]^d)$ will always be denoted by the corresponding lowercase letters $p,q,r,\dots$, and vice versa. Note that checkerboard copulas are simply the d.f.s of the probability measures in $\check \Cc_n([0,1]^d)$.



We can now define what we mean by checkerboard approximation of a $d$-stochastic measure. Let $P \in \Cc([0,1]^d)$ and notice that $p \in \Ac_{d,n}$ defined by $p_{\bm i} := P(B_{\bm i})$, $\bm i \in [n]^d$, is a copula array, that is, $p \in \Cc_{d,n}$. The checkerboard approximation of $P$ is then simply $\check P \in \check \Cc_n([0,1]^d)$ given by~\eqref{eq:check:P}. We end this section by mentioning an important property which justifies using checkerboard approximations. Let $\check C_n$ and $C$ be the d.f.s (that is, the copulas) of $\check P$ and $P$, respectively. Then, as verified for instance in \cite[proof of Theorem~4.1.5]{DurSem15},
\begin{equation}
  \label{eq:approx}
  \sup_{\bm v \in [0,1]^d} | \check C_n(\bm v) - C(\bm v) | \leq \frac{d}{n}.
\end{equation}
Letting $n$ tend to $\infty$ (only this one time), the latter inequality immediately implies that the sequence of checkerboard approximations of $C$ converges uniformly to $C$. Note that the sequence of checkerboard approximations of $C$ also converges to $C$ in a stronger sense; see, e.g., \cite[Theorem~2]{LiMikSheTay97} or \cite[Corollary~3.2]{LiMikTay98}.



\subsection{$I$-projections}
\label{sec:Iproj}

The following definition is due to \citet[Section~1]{Csi75}.

\begin{defn}[$I$-projection]
  \label{defn:Iproj}
  Let $T \in \Mc(\R^d)$, let $D \subset \Mc(\R^d)$ be a convex set of probability measures and assume that there exists $P \in D$ such that $I(P \| T) < \infty$. Then $S \in D$ satisfying $I(S \| T) = \min_{P \in D} I(P \| T)$ is called the $I$-projection of $T$ on $D$.
\end{defn}

As remarked in \cite[Section~1]{Csi75}, the existence of an $I$-projection guarantees its uniqueness. Since, as discussed in Section~\ref{sec:arrays}, probability arrays in $\Pc_{d,n}$ can be seen as encoding the values of the p.m.f.s of probability measures in $\Mc([n]^d)$, the notion of $I$-projection can be easily extended to probability arrays. First, we need to formally define the Kullback--Leibler divergence for probability arrays. For any $p,q \in \Pc_{d,n}$, it is easy to verify that $\supp(p) \subset \supp(q) \iff \un{P} \ll \un{Q}$. Starting from~\eqref{eq:KL}, it is then natural to define the Kullback--Leibler divergence of $p \in \Pc_{d,n}$ with respect to $q \in \Pc_{d,n}$ as
\begin{equation}
  \label{eq:KL:discrete}
   I(p \| q) := I(\un{P} \| \un{Q}) = \begin{cases}
    \displaystyle \sum_{\bm i \in [n]^d} p_{\bm i} \log \frac{p_{\bm i}}{q_{\bm i}}, &\text{if } \supp(p) \subset \supp(q), \\
    \infty, & \text{otherwise},
  \end{cases}
\end{equation}
with the conventions that $0 \log 0 := 0$ and $0 \log(0/0) := 0$. We then adopt the following natural definition.

\begin{defn}[$I$-projection for probability arrays]
  \label{defn:Iproj:prob:array}
  Let $t \in \Pc_{d,n}$, let $\Dc \subset \Pc_{d,n}$ be a convex set of probability arrays and assume that there exists $p \in \Dc$ such that $I(p \| t) < \infty$. Then $s \in \Dc$ satisfying $I(s \| t) = \min_{p \in \Dc} I(p \| t)$ will be called the $I$-projection of $t$ on $\Dc$.
\end{defn}

\begin{remark}
  \label{rem:extension}
  More generally, using the aforementioned one-to-one correspondence between $\Mc([n]^d)$ and $\Pc_{d,n}$, in the rest of this work, all the terminology and results holding for probability measures in $\Mc([n]^d)$ will be implicitly extended to probability arrays. \qed
\end{remark}


\section{The generalized minimum information copula problem and its checkerboard approximation}
\label{sec:MIC}

The aim of this section is to introduce and study a generalization of the minimum information copula problem as well as its checkerboard version. We first explicitly state the version of the minimum information copula problem in~\eqref{eq:MIC} studied in \cite{PiaHowBor12}. We then introduce a more general version of this problem and provide conditions under which it has a unique solution. Next, since the aforementioned problem is not tractable in general, we consider its version with all the probability measures involved in its formulation replaced by checkerboard approximations as defined in Section~\ref{sec:approx}. This is very similar to what was carried out in \cite{PiaHowBor12,BedWil14,SukSei25a}. Finally, we verify that the resulting generalized minimum information checkerboard copula problem is a particular instance of the so-called discrete $I$-projection linear problem, which will allow us to algorithmically solve it in Section~\ref{sec:solving}.

\subsection{The minimum information copula problem under fixed Spearman's rhos}

Recall the formulation of the minimum information copula problem in~\eqref{eq:MIC}. In \cite{PiaHowBor12}, $M = d(d-1)/2$ and each of the $M$ expectation constraints corresponds to the desired value of Spearman's rho for the $\{i,j\}$-margin of $P$, $\{i,j\} \subset [d]$. Let $\rho$ be the function from $\Mc([0,1]^2)$ to $[-1,1]$ defined by
\begin{equation}
  \label{eq:rho}
  \rho(P) := \int_{[0,1]^2} g_\rho(\bm v) \dd P(\bm v), \qquad P \in \Mc([0,1]^2),
\end{equation}
where
\begin{equation}
  \label{eq:g:rho}
  g_\rho(\bm v) := 12 \left(v_1 - \frac{1}{2} \right) \left(v_2 - \frac{1}{2}\right), \qquad \bm v \in [0,1]^2.
\end{equation}
Then, for any $P \in \Cc([0,1]^2)$, it can be verified that $\rho(P)$ is Spearman's rho of the copula of $P$ \cite[see, e.g.,][Section~2.6 and the references therein]{HofKojMaeYan18}. Using the previous notation, the problem addressed in \cite{PiaHowBor12} can be rewritten as
\begin{equation}
  \label{eq:MICS}
  \begin{split}
    \min_{P \in \Cc([0,1]^d)} &\, I(P \| U_d) \text{ subject to } \\
    &\rho(P^{(\{i,j\})}) = \alpha_{\{i,j\}}, \{i,j\} \subset [d],
  \end{split}
\end{equation}
for some real numbers $\alpha_{\{i,j\}} \in [-1,1]$, $\{i,j\} \subset [d]$. Following \citet{SukSei25a}, we call Problem~\eqref{eq:MICS} the minimum information copula problem under fixed Spearman's rhos. Roughly speaking, the aim is to find the closest copula to the independence copula (in terms of the Kullback--Leibler divergence) that has the specified Spearman's rhos, if it exists.

\subsection{The generalized minimum information copula problem}
\label{sec:GMIC}

Let $\Jc, \Kc \subset 2^{[d]}$ such that, for any $J \in \Jc$, $|J| \geq 2$, for any $K \in \Kc$, $|K| \geq 2$, and $\Jc \cap \Kc = \emptyset$. Furthermore, for any subset $K \in \Kc$, let $G_K$ be a function from $\Mc([0,1]^{|K|})$ to $\R$ defined by
\begin{equation}
  \label{eq:G:K}
  G_K(P) := \int_{[0,1]^{|K|}} g_K(\bm v) \dd P(\bm v), \qquad P \in \Mc([0,1]^{|K|}),
\end{equation}
for some continuous function $g_K:[0,1]^{|K|} \to \R$. Let $R \in \Cc([0,1]^d)$ and $S^J \in \Cc([0,1]^{|J|})$, $J \in \Jc$, be known probability measures. Furthermore, let $\alpha_K$, $K \in \Kc$, be $|\Kc|$ given real numbers. We call generalized minimum information copula problem the following optimization problem:
\begin{equation}
  \label{eq:GMIC}
  \begin{split}
    \min_{P \in \Cc([0,1]^d)} \,& I(P \| R) \text{ subject to } \\
    &P^{(J)} = S^J, J \in \Jc, \\
    &G_K(P^{(K)}) = \alpha_K, K \in \Kc.
  \end{split}
\end{equation}
Roughly speaking, the aim is to find the closest copula to the copula of $R$ (in terms of the Kullback--Leibler divergence) that satisfies the specified marginal constraints, if it exists. For $P \in \Cc([0,1]^{|K|})$ and $|K| = 2$, $G_K(P)$ would typically correspond to a moment of the copula of $P$ such as Spearman's rho, Gini's gamma, etc, that can be written as an expectation with respect to $P$ \citep[see, e.g.,][and the references therein]{Lie14}. A very natural choice for $R$ is $U_d$, the probability measure of the uniform distribution on $[0,1]^d$, as the problem can then be interpreted as a maximum entropy principle.

\begin{example}
  \label{ex:costal}
  Although setting $R = U_d$ is often a natural choice, some other choices may be relevant in specific situations. To illustrate this, let us consider a fictitious scenario in coastal engineering. Assume that the aim is to build a coastal protection (such as a dike) at Location A on a coast and that a similar coastal protection already exists at Location~B a few tens of kilometers away. Note that to properly dimension a coastal protection, one typically needs to know the joint distribution of certain meteorological and sea-state variables at the location of interest. Let $\bm X_A$ and $\bm X_B$ be the corresponding random vectors for Locations $A$ and $B$, respectively. Because of the relative proximity of Locations A and B, to determine the copula of $\bm X_A$ one could decide to use~\eqref{eq:GMIC} with $R$ equal to (the already determined/estimated) probability measure of the copula of $\bm X_B$. One could then interpret~\eqref{eq:GMIC} as attempting to find the closest copula to that of $X_B$ that satisfies certain ``specialized'' constraints for Location A. \qed
\end{example}

\begin{remark}
  \label{rem:OT}
  Problem~\eqref{eq:GMIC} bears some similarity with entropy-regularized optimal transport problems with additional linear constraints \cite[see, e.g.,][and the references therein]{EckKup21}, with the exception that, to the best of our knowledge, the latter do not involve higher-order, possibly overlapping, marginals constraints. Such constraints arise in the so-called marginal compatibility problem which is concerned with whether a joint distribution with the specified margins exists; see, e.g., \cite{Rus91} for the continuous case, \cite{Joe97,DurKleQue08} for the specific case of copulas and \cite{FriCha13} for the discrete case. \qed
\end{remark}

Next, let
\begin{align}
  \label{eq:F:J}
  F_J &:= \{P \in \Mc([0,1]^d) : P^{(J)} = S^J \}, \qquad J \in \Jc, \\
  \label{eq:L:K}
  L_K &:= \{ P \in \Mc([0,1]^d) : G_K(P^{(K)}) = \alpha_K \}, \qquad K \in \Kc, \\
  \label{eq:E}
  E &:= \Cc([0,1]^d) \cap \bigcap_{J \in \Jc} F_J \cap \bigcap_{K \in \Kc} L_K.
\end{align}
Problem~\eqref{eq:GMIC} can then be compactly reformulated as $\min_{P \in E} I(P \| R)$.  When attempting to solve it, we can distinguish between three mutually exclusive scenarios:
\begin{description}
\item[\bf Case 1:] The set $E$ is empty or, equivalently, the constraints in~\eqref{eq:GMIC} are inconsistent.

\item[\bf Case 2:] The constraints in~\eqref{eq:GMIC} are consistent, that is, there exists a $P \in E$, but it is not possible to choose $P$ such that $I(P \| R) < \infty$.

\item[\bf Case 3:] The constraints in~\eqref{eq:GMIC} are consistent, that is, there exists a $P \in E$, and it is possible to choose $P$ such that $I(P \| R) < \infty$.

\end{description}

In Case~1, Problem~\eqref{eq:GMIC} has no solution, obviously. In Case~2, Problem~\eqref{eq:GMIC} does not have ``interesting'' solutions as the objective function is always $\infty$. We focus on Case~3 hereafter. The following result can then be stated. It is proven in Appendix~\ref{proofs:MIC}.

\begin{prop}
  \label{prop:unicity}
  Assume that there exists $P \in E$ such that $I(P \| R) < \infty$. Then, Problem~\eqref{eq:GMIC} admits a unique solution $Q \in E$ satisfying $P \ll Q \ll R$ for all $P \in E$ such that $I(P \| R) < \infty$.
\end{prop}

\begin{remark}
  The set $E$ can be easily verified to be convex. Hence, following Definition~\ref{defn:Iproj}, under the assumption of Proposition~\ref{prop:unicity}, the unique solution $Q$ of Problem~\eqref{eq:GMIC} is the $I$-projection of $R$ on $E$. \qed
\end{remark}

\begin{remark}
  Among the ingredients of Problem~\eqref{eq:GMIC}, one finds the $|\Jc|$ probability measures $S^J \in \Cc([0,1]^{|J|})$, $J \in \Jc$, defining some of the marginal constraints. When the constraints are consistent (see Case~2 or Case~3 above), there exists $P \in E$, which implies that, for any $J \in \Jc$, $P^{(J)} = S^J$. It follows that in Case~2 or Case~3 above, the probability measures $S^J \in \Cc([0,1]^{|J|})$, $J \in \Jc$, can be regarded as the higher-order margins of the same probability measure in $\Cc([0,1]^d)$. \qed
\end{remark}

\begin{remark}
  Following for instance \cite{Csi75,CsiShi04}, any intersection of the sets $L_K$ in~\eqref{eq:L:K} will be called a linear family of probability measures. From the previous remark, any intersection of the sets $F_J$ in~\eqref{eq:F:J} can be called a Fréchet class of probability measures. \qed
\end{remark}


\subsection{Checkerboard approximation of the problem}

Let $n \geq 2$ be a fixed discretization parameter as defined in Section~\ref{sec:approx}. Recall the formulation of the generalized minimum information copula problem given in~\eqref{eq:GMIC} and let $\check R \in \check \Cc_n([0,1]^d)$ and $\check S^J \in \check \Cc_n([0,1]^{|J|})$, $J \in \Jc$, be the checkerboard approximations of $R \in \Cc([0,1]^d)$ and $S^J \in \Cc([0,1]^{|J|})$, $J \in \Jc$, respectively. Since Problem~\eqref{eq:GMIC} is intractable in general, in the spirit of \cite{PiaHowBor12,BedWil14,SukSei25a}, among others, we consider its version with all the probability measures involved in its formulation replaced by checkerboard approximations. Roughly speaking, following~\eqref{eq:approx}, we would ideally want to set the discretization parameter $n$ to be ``as large as possible''. This aspect is discussed in more detail in Section~\ref{sec:num} and in the supplement \cite{KojMar25}. We call the resulting problem the generalized minimum information checkerboard copula problem:
\begin{equation}
  \label{eq:GMICC}
  \begin{split}
    \min_{\check P \in \check \Cc_n([0,1]^d)} \,& I(\check P \| \check R) \text{ subject to } \\
                                                &\check P^{(J)} = \check S^J, J \in \Jc, \\
                                                &G_K(\check P^{(K)}) = \alpha_K, K \in \Kc.
  \end{split}
\end{equation}

Recall from Section~\ref{sec:approx} that $\check \Cc_n([0,1]^d)$ is in one-to-one correspondence with the set of copula arrays $\Cc_{d,n}$. We can thus expect to be able to completely formulate Problem~\eqref{eq:GMICC} in terms of copula arrays. For any $\check P, \check Q \in \check \Cc_n([0,1]^d)$, it can be verified from~\eqref{eq:check:P} that $\check P \ll \check Q \iff \supp(p) \subset \supp(q)$ and then, from~\eqref{eq:KL}, that
\begin{equation*}
  I(\check P \| \check Q) = \begin{cases}
    \displaystyle \sum_{\bm i \in [n]^d} p_{\bm i} \log \frac{p_{\bm i}}{q_{\bm i}}, &\text{if } \supp(p) \subset \supp(q), \\
    \infty, & \text{otherwise}.
  \end{cases}
   = I(p \| q),
\end{equation*}
where $I(p \| q)$ is defined in~\eqref{eq:KL:discrete}. Furthermore, for any $\check P \in \check \Cc_n([0,1]^d)$ and $K = \{\ell_1, \dots, \ell_{|K|}\} \in \Kc$, $1 \leq \ell_1 < \dots < \ell_{|K|} \leq d$, from~\eqref{eq:G:K} and~\eqref{eq:check:p},
\begin{align*}
  G_K(\check P^{(K)}) &= \int_{[0,1]^{|K|}} g_K(\bm v) \dd \check P^{(K)}(\bm v) = \int_{[0,1]^{|K|}} g_K(\bm v) n^{|K|} \left( \sum_{\bm i_K \in [n]^{|K|}} \1_{\prod_{j=1}^{|K|} A_{i_{\ell_j}}}(\bm v) p_{\bm i_K}^{(K)} \right) \dd \bm v \\
                    &= \sum_{\bm i_K \in [n]^{|K|}} p_{\bm i_K}^{(K)} n^{|K|} \int_{[0,1]^{|K|}} g_K(\bm v) \1_{\prod_{j=1}^{|K|} A_{i_{\ell_j}}}(\bm v) \dd \bm v = \sum_{\bm i \in [n]^d} p_{\bm i}  h^K_{\bm i},
\end{align*}
where the sets $A_i$ are defined in Section~\ref{sec:approx} and $h^K$ is the array in $\Ac_{d,n}$ defined by
\begin{equation}
  \label{eq:h:K}
  h^K_{\bm i} = n^{|K|} \int_{\prod_{j=1}^{|K|} A_{i_{\ell_j}}} g_K(\bm v) \dd \bm v, \qquad \bm i \in [n]^d.
\end{equation}
Finally, using the fact that checkerboard probability measures are equal if and only if their skeletons are equal, Problem~\eqref{eq:GMICC} can be reformulated as
\begin{equation}
  \label{eq:GMICC:p:0}
  \begin{split}
    \min_{p \in \Cc_{d,n}} \,& I(p \| r) \text{ subject to } \\
                             &p^{(J)} = s^J, J \in \Jc, \\
                             &\sum_{\bm i \in [n]^d} p_{\bm i}  h^K_{\bm i} = \alpha_K, K \in \Kc,
  \end{split}
\end{equation}
where $r \in \Cc_{d,n}$ and the $s^J \in \Cc_{|J|,n}$, $J \in \Jc$, are the skeletons of $\check R \in \check \Cc_n([0,1]^d)$ and $\check S^J \in \check \Cc_n([0,1]^{|J|})$, $J \in \Jc$, respectively.


\subsection{Reformulation in terms of probability arrays}
\label{sec:sol}

With the aim of solving the generalized minimum information checkerboard copula problem, we are going to provide a straightforward reformulation of it in terms of probability arrays (and not solely copula arrays). Let
\begin{equation}
  \label{eq:Jc'}
  \Jc' := \Jc \cup \{\{1\}, \dots, \{d\}\}
\end{equation}
and let $s^{\{1\}}, \dots, s^{\{d\}} \in \Cc_{1,n}$ be equal to the univariate uniform probability array $u_1 \in \Cc_{1,n}$, that is,
\begin{equation}
  \label{eq:s:ell}
  s^{\{\ell\}} :=  u_1, \qquad \ell \in [d].
\end{equation}
Next, let
\begin{align}
  \label{eq:Fc:J}
  \Fc_J &:= \left\{p \in  \Pc_{d,n} : p^{(J)} = s^J \right\}, \qquad J \in \Jc',\\
  \label{eq:Lc:K}
  \Lc_K &:= \left\{ p \in  \Pc_{d,n} : \sum_{\bm i \in [n]^d} p_{\bm i}  h^K_{\bm i} = \alpha_K \right\}, \qquad K \in \Kc. \\
  \label{eq:Ec}
  \Ec &:= \bigcap_{J \in \Jc'} \Fc_J \cap \bigcap_{K \in \Kc} \Lc_K.
\end{align}
Then, using the fact that $\Cc_{d,n} = \bigcap_{\ell \in [d]} \Fc_{\{\ell\}}$ and the definition of $\Jc'$ in~\eqref{eq:Jc'}, it is easy to verify that Problem~\eqref{eq:GMICC:p:0} can be reformulated in terms of probability arrays as
\begin{equation}
  \label{eq:GMICC:p}
  \begin{split}
    \min_{p \in \Pc_{d,n}} \,& I(p \| r) \text{ subject to } \\
                             &p^{(J)} = s^J, J \in \Jc', \\
                             &\sum_{\bm i \in [n]^d} p_{\bm i}  h^K_{\bm i} = \alpha_K, K \in \Kc,
  \end{split}
\end{equation}
or, more compactly, as $\min_{p \in \Ec} I(p \| r)$. When attempting to solve it, we proceed exactly as in Section~\ref{sec:GMIC} and consider three mutually exclusive possibilities. If $\Ec = \emptyset$, the constraints are inconsistent and Problem~\eqref{eq:GMICC:p} has no solution. If $\Ec \neq \emptyset$ but there is no $p \in \Ec$ such that $\supp(p) \subset \supp(r)$, Problem~\eqref{eq:GMICC:p} has no ``interesting'' solutions since, following~\eqref{eq:KL:discrete}, $I(p \| r) = \infty$ for all $p \in \Ec$. We naturally focus on the analog of Case~3 in Section~\ref{sec:GMIC} which is equivalent to working under the following condition:

\begin{cond}
  \label{cond:included}
  There exists $p \in \Ec$ in~\eqref{eq:Ec} such that $\supp(p) \subset \supp(r)$.
\end{cond}

The following result is then the analog of Proposition~\ref{prop:unicity}. It is merely a consequence of the fact that $\Ec$ is convex and closed (when seen as a subset of $\R^{n^d}$), and that, as discussed in Section~\ref{sec:arrays}, probability arrays in $\Pc_{d,n}$ can be regarded as encoding the p.m.f.s of probability measures in $\Mc([n]^d)$, as well as Theorem~2.1 and the remark following Theorem~2.2 in \cite{Csi75}.

\begin{prop}
  \label{prop:unicity:p}
  Assume that Condition~\ref{cond:included} holds. Then, Problem~\eqref{eq:GMICC:p} admits a unique solution $q \in \Ec$ satisfying $\supp(p) \subset \supp(q) \subset \supp(r)$ for all $p \in \Ec$ such that $\supp(p) \subset \supp(r)$.
\end{prop}

\begin{remark}
  \label{rem:check}
  Obviously, the unique solution $q$ of the generalized minimum information checkerboard copula problem under Condition~\ref{cond:included} can also naturally be expressed as $\check Q \in \check \Cc_n([0,1]^d)$, the checkerboard copula with skeleton $q$. \qed
\end{remark}

\begin{remark}
  Following Remark~\ref{rem:extension}, any intersection of the $\Fc_J$ in~\eqref{eq:Fc:J} will be called a Fréchet class of probability arrays and any intersection of the $\Lc_K$ in~\eqref{eq:Lc:K} will be called a linear family of probability arrays. As we shall see in the next subsection, Fréchet classes of probability arrays are actually particular linear families of probability arrays. Finally, following Definition~\ref{defn:Iproj:prob:array}, 
  the unique solution $q$ of Problem~\eqref{eq:GMICC:p} under Condition~\ref{cond:included} in Proposition~\ref{prop:unicity:p} is the $I$-projection of $r$ on $\Ec$. \qed
\end{remark}


\subsection{An instance of the discrete $I$-projection linear problem}
\label{sec:DIPL}

Let $t \in \Pc_{d,n}$, let $h_1,\dots,h_b$ be $b \geq 1$ arrays in $\Ac_{d,n}$, let $a_1,\dots,a_b \in \R$ and let
\begin{equation}
  \label{eq:Ec':init}
  \Ec' := \bigcap_{k \in [b]} \left\{p \in \Pc_{d,n} : \sum_{\bm i \in [n]^d} p_{\bm i} h_{k,\bm i} = a_k \right\}.
\end{equation}
The generic problem $\min_{p \in \Ec'} I(p \| t)$ (which could be easily reformulated using discrete probability measures in $\Mc([n]^d)$ -- see Section~\ref{sec:arrays}) has been extensively studied in the literature; see for instance \cite{DarRat72},  \cite[Section 3]{Csi75}, \cite{Csi89}, \cite[Chapter~5]{CsiShi04} and \cite[Chapter~15]{PolWu24}. For ease of reference, we call it the discrete $I$-projection linear problem in the rest of this work.


We shall now verify that Problem~\eqref{eq:GMICC:p} is a particular instance of the discrete $I$-projection linear problem. For any $J \in \Jc'$ and $i_J^* \in [n]^{|J|}$, let $h^{\bm i_J^*}$ be the array in $\Ac_{d,n}$ defined by $h^{\bm i_J^*}_{\bm i} := \1_{\{\bm i \in [n]^d: \bm i_J = \bm i_J^*\}}(\bm i)$, $\bm i \in [n]^d$. Then, for any $J \in \Jc'$,
\begin{align*}
    p^{(J)} = s^J &\iff p^{(J)}_{\bm i_J^*}  = s^J_{\bm i_J^*}, \qquad \forall \bm i_J^* \in [n]^{|J|}, \\
                      &\iff \sum_{\bm i \in [n]^d} p_{\bm i} \1_{\{\bm i \in [n]^d: \bm i_J = \bm i_J^*\}}(\bm i)  = s^J_{\bm i_J^*}, \qquad \forall \bm i_J^* \in [n]^{|J|}, \\
&\iff \sum_{\bm i \in [n]^d} p_{\bm i} h^{\bm i_J^*}_{\bm i}  = s^J_{\bm i_J^*}, \qquad \forall \bm i_J^* \in [n]^{|J|}.
  \end{align*}
  In other words, for any $J \in \Jc'$, the constraint $p^{(J)} = s^J$ can be reformulated as $n^{|J|}$ expectation constraints. Note that when $J = \{\ell\}$ for some $\ell \in [d]$, as a consequence of~\eqref{eq:s:ell},
$$
  p^{(\{\ell\})} =  s^{\{\ell\}} \iff \sum_{\bm i \in [n]^d} p_{\bm i} \1_{\{\bm i \in [n]^d: \bm i_\ell = \bm i_\ell^*\}}(\bm i) = \frac{1}{n}, \qquad \forall i_\ell^* \in [n].
$$
Hence, all the constraints in Problem~\eqref{eq:GMICC:p} can be rewritten as expectation constraints and Problem~\eqref{eq:GMICC:p} is indeed a particular instance of the discrete $I$-projection linear problem.


\section{Solving the generalized minimum information checkerboard copula problem via an iterated $I$-project\-ion procedure}
\label{sec:solving}

From Section~\ref{sec:sol}, we know that the generalized minimum information checkerboard copula problem in~\eqref{eq:GMICC} can be compactly written as $\min_{p \in \Ec} I(p \| r)$, where $\Ec$ is defined in~\eqref{eq:Ec} and $r \in \Cc_{d,n}$ is the skeleton of $\check R \in \check \Cc_n([0,1]^d)$. As also explained therein, we wish to solve this problem under Condition~\ref{cond:included}, which, following Proposition~\ref{prop:unicity:p} and Definition~\ref{defn:Iproj:prob:array}, amounts to finding the $I$-projection $q$ of $r$ on $\Ec$. Since, from Section~\ref{sec:DIPL}, $\min_{p \in \Ec} I(p \| r)$ is a particular instance of the discrete $I$-projection linear problem, Theorem~3.2 of \citet{Csi75} immediately implies that under Condition~\ref{cond:included} there exists a generic iterated $I$-projection procedure for approximately finding $q$. The aim of this section is to explain how this procedure can be made fully operational. Because of the underlying discrete finite setting, all the $I$-projections mentioned in this section will be described in terms of probability arrays.

\subsection{The iterated $I$-projection procedure}
\label{sec:algo}

Recall the definition of $\Jc'$ in~\eqref{eq:Jc'} and let $N := |\Jc'| + |\Kc|$. Next, let us arbitrarily rename the $\Fc_J$ in~\eqref{eq:Fc:J} as $\Ec_1,\dots,\Ec_{|\Jc'|}$ and the $\Lc_K$ in~\eqref{eq:Lc:K} as $\Ec_{|\Jc'|+1},\dots,\Ec_{N}$. Since Problem~\eqref{eq:GMICC} formulated as $\min_{p \in \Ec} I(p \| r)$ is a particular instance of the discrete $I$-projection linear problem, under Condition~\ref{cond:included}, from Theorem 3.2 of \cite{Csi75}, there exists a generic procedure based on successive $I$-projections on $\Ec_1,\dots,\Ec_N$ whose result converges to the $I$-projection $q$ of $r$ on $\Ec$. Specifically, let $q^{[0]} := r$ and, for any $m \geq 1$, let $q^{[m]}$ be the $I$-projection of $q^{[m-1]}$ on $\Ec_{\change{((m-1) \bmod N) + 1}}$. Then, Theorem 3.2 of \cite{Csi75} guarantees that $q^{[m]}$ converges to $q$ as $m \to \infty$.

The above immediately suggests the following algorithm for computing an approximation of the $I$-projection $q$ of $r$ on $\Ec$ in~\eqref{eq:Ec} under Condition~\ref{cond:included}.

\begin{algorithm}[H]
  \label{algo:iterated:Iproj}
  \caption{Iterated $I$-projection procedure for solving problem $\min_{p \in \Ec} I(p \| r)$.}  
  \SetKwInOut{Input}{Input}
  \SetKwInOut{Output}{Output}
  \Input{An input array $r \in \Pc_{d,n}$, a small, strictly positive real number $\varepsilon$ to be used in the stopping condition (see below), a maximum number of iterations $M$ and the linear sets $\Fc_J$, $J \in \Jc'$, and $\Lc_K$, $K \in \Kc$, whose intersection is equal to $\Ec$ in~\eqref{eq:Ec}.}
  \Output{The approximation $q^{[m]}$ of the $I$-projection $q$ of $r$ on $\Ec$.}
  $q^{[0]} = r$  \\
  $m = 0$ \\
  \For{$iter = 1, \dots, M$}{
    \For{$J \in \Jc'$}{
      $m = m + 1$ \\
      Compute the $I$-projection $q^{[m]}$ of $q^{[m-1]}$ on $\Fc_J$. \label{line:Iproj:Fc}
    }
    \For{$K \in \Kc$}{
      $m = m + 1$ \\
      Compute the $I$-projection $q^{[m]}$ of $q^{[m-1]}$ on $\Lc_K$. \label{line:Iproj:Gc}
    }
    \If{$\max_{\bm i \in [n]^d} |q^{[m]}_{\bm i} - q^{[m-|\Jc'|-|\Kc|]}_{\bm i}| < \varepsilon$ \label{line:num:conv}}
    {
      Exit the loop and print that numerical convergence has been reached.
    }
  }
\end{algorithm}

As one can see from Algorithm~\ref{algo:iterated:Iproj}, the idea is to successively $I$-project on the $\Fc_J$, $J \in \Jc'$, and the $\Lc_K$, $K \in \Kc$, until the current approximation of the $I$-projection of $r$ on $\Ec$ in~\eqref{eq:Ec} stabilizes, that is, changes by less than $\eps$ elementwise. 

\subsection{$I$-projections on the $\Fc_J$}

To make Algorithm~\ref{algo:iterated:Iproj} operational, we need to make its Lines~\ref{line:Iproj:Fc} and~\ref{line:Iproj:Gc} operational, that is, we need to be able to compute $I$-projections on each $\Fc_J$ in~\eqref{eq:Fc:J} and on each $\Lc_K$ in~\eqref{eq:Lc:K}.  Let us start with the former. The following corollary is a consequence of the first result mentioned in \cite[Section~5.1]{CsiShi04}. For completeness, a statement of the latter, along with its proof and the proof of the corollary, are given in Appendix~\ref{proofs:solving}.

\begin{cor}
  \label{cor:IPFP:any:marg}
  Fix $J \in \Jc'$ and let $q^\dagger \in \Pc_{d,n}$ such that $\supp (s^J) \subset \supp(q^{\dagger, (J)})$. Then, the $I$-projection $q^\star$ of $q^\dagger$ on $\Fc_J = \left\{p \in \Pc_{d,n}:  p^{(J)}  = s^J \right\}$ exists, is unique and is given by
  \begin{equation*}
    q^\star_{\bm i} :=
    \begin{cases}
      q^\dagger_{\bm i}  \frac{s^J_{\bm i_J}}{q^{\dagger,(J)}_{\bm i_J}}, &\qquad \text{if } \bm i \in \supp(q^\dagger), \\
      0, &\qquad \text{otherwise}.
    \end{cases}
  \end{equation*}
\end{cor}

\subsection{$I$-projections on the $\Lc_K$ using generalized iterative scaling}
\label{sec:operational:Gc:GIS}

We shall now explain how to make Line~\ref{line:Iproj:Gc} of Algorithm~\ref{algo:iterated:Iproj} operational. The first approach that we consider consists of using generalized iterative scaling \citep[see, e.g.,][]{DarRat72,Csi89}. Under suitable conditions, the latter technique can actually be used to obtain an approximation of an $I$-projection on any intersection of the $\Lc_K$ in~\eqref{eq:Lc:K}. We present it for an arbitrary linear set $\Ec'$ of probability arrays as defined in~\eqref{eq:Ec':init} in terms of $b \geq 1$ arrays $h_1,\dots,h_b$ in $\Ac_{d,n}$ and $a_1,\dots,a_b \in \R$. Then, as noticed in Lemma 4 of Section~1 of \cite{DarRat72} \citep[see also][Section~5.1]{Csi89,CsiShi04}, there exists $c \geq b$ nonnegative arrays $\bar h_1,\dots,\bar h_c \in \Ac_{d,n}$ satisfying
\begin{equation}
  \label{eq:h:GIS}
  \sum_{k = 1}^c  \bar h_{k,\bm i} = 1 \qquad \text{for all } \bm i \in [n]^d,
\end{equation}
and a probability vector $(\bar a_1,\dots,\bar a_c)$ such that $\Ec'$ in~\eqref{eq:Ec':init} can be equivalently rewritten as
\begin{equation}
  \label{eq:Ec':GIS}
\Ec' = \bigcap_{k \in [c]} \left\{p \in \Pc_{d,n} : \sum_{\bm i \in [n]^d} p_{\bm i} \bar h_{k,\bm i} = \bar a_k \right\}.
\end{equation}

An algorithm to compute the arrays $\bar h_1,\dots,\bar h_c$ and $\bar a_1,\dots,\bar a_c$ from the initial formulation of $\Ec'$ in~\eqref{eq:Ec':init} is for instance given in \cite[Section 3.2]{LinSte23} \citep[see also the discussion in][before Theorem~5.2]{CsiShi04}. The following theorem is then the main result of \cite{Csi89} \citep[see also][Theorem 5.2]{CsiShi04}.

\begin{thm}
  \label{thm:GIS}
  Let $q^\dagger \in \Pc_{d,n}$ and assume that there exists $p \in \Ec'$ such that $\supp(p) \subset \supp(q^\dagger)$. Then, the $I$-projection $q^\star$ of $q^\dagger$ on $\Ec'$ exists and is unique. Furthermore, let $q^{\star,[0]} := q^\dagger$ and, for any $m \geq 1$, let
  \begin{equation}
    \label{eq:update:GIS}
    q^{\star,[m]}_{\bm i} := q^{\star,[m-1]}_{\bm i} \prod_{k = 1}^c \left( \frac{\bar a_k}{\sum_{\bm i' \in [n]^d} q^{\star,[m-1]}_{\bm i'} \bar h_{k, \bm i'}} \right)^{\bar h_{k, \bm i}}, \qquad \bm i \in [n]^d,
  \end{equation}
  with the conventions that $0/0 := 0$ and $0^0 := 1$. Then, $q^{\star,[m]}$ converges to $q^\star$ as $m \to \infty$.
\end{thm}

The previous theorem immediately translates into the following algorithm:

\begin{algorithm}[H]
  \label{algo:GIS}
  \caption{Generalized iterative scaling for computing the $I$-projection of $q^\dagger$ on $\Ec'$ in~\eqref{eq:Ec':GIS}} 

  \SetKwInOut{Input}{Input}
  \SetKwInOut{Output}{Output}
  \Input{An input array $q^\dagger \in \Pc_{d,n}$, a small, strictly positive real number $\varepsilon'$ to be used in the stopping condition (see below), a maximum number of iterations $M'$ as well as $\Ec'$ given as in~\eqref{eq:Ec':GIS}.}
  \Output{The approximation $q^{\star,[m]}$ of the $I$-projection $q^\star$ of $q^\dagger$ on $\Ec'$.}
  $q^{\star,[0]} = q^\dagger$  \\
  \For{$m = 1, \dots, M'$}{
    Compute $q^{\star,[m]}$ from $q^{\star,[m-1]}$ using~\eqref{eq:update:GIS}. \\
    \If{$\max_{\bm i \in [n]^d} |q^{\star,[m]}_{\bm i} - q^{\star,[m-1]}_{\bm i}| < \varepsilon'$}
    {
      Exit the loop and print that numerical convergence has been reached.
    }
  }
\end{algorithm}

Note that Algorithm~\ref{algo:GIS} can be viewed as a particular case of the SMART algorithm described for instance in \cite{Byr93} and employed in image processing. Clearly, it can be directly used for attempting to solve any discrete $I$-projection linear problem, and thus in particular the generalized minimum information checkerboard copula problem (see Section~\ref{sec:DIPL}). In the latter case, it would suffice to express $\Ec$ in~\eqref{eq:Ec} in the required form (that is, in terms of nonnegative arrays satisfying~\eqref{eq:h:GIS}) and apply Algorithm~\ref{algo:GIS} with $q^\dagger = r$. However, as argued in \cite{LinSte23}, this is likely to be slower than the use of iterated $I$-projections as in Algorithm~\ref{algo:iterated:Iproj}, especially when many of the intermediate $I$-projections can be computed via simple scalings of the form of those given in Corollary~\ref{cor:IPFP:any:marg}.

In order to use generalized iterative scaling to make Line~\ref{line:Iproj:Gc} of Algorithm~\ref{algo:iterated:Iproj} operational, we need to specialize~\eqref{eq:update:GIS} in Theorem~\ref{thm:GIS}. The following lemma is proven in Appendix~\ref{proofs:solving}.


\begin{lem}
  \label{lem:Gc:GIS}
  If $\Ec'$ in Theorem~\ref{thm:GIS} is equal to $\Lc_K$ in~\eqref{eq:Lc:K} for some $K \in \Kc$, a suitable specialization of~\eqref{eq:update:GIS} is
  \begin{equation}
    \label{eq:update:GIS:Gc}
    q^{\star,[m]}_{\bm i} := q^{\star,[m-1]}_{\bm i} \left( \frac{\bar a_K}{\sum_{\bm i' \in [n]^d} q^{\star,[m-1]}_{\bm i'} \bar h^K_{\bm i'}} \right)^{\bar h^K_{\bm i}} \left( \frac{1 - \bar a_K}{\sum_{\bm i' \in [n]^d} q^{\star,[m-1]}_{\bm i'} (1 - \bar h^K_{\bm i'})} \right)^{1 - \bar h^K_{\bm i}}, \, \bm i \in [n]^d,
  \end{equation}
  with the conventions that $0/0 := 0$ and $0^0 := 1$, where
  $$
  \bar h^K_{\bm i} := \frac{h^K_{\bm i} - \delta_K}{\Delta_K - \delta_K}, \qquad \bm i \in [n]^d, \qquad \bar a_K := \frac{\alpha_K - \delta_K}{\Delta_K - \delta_K},
  $$
  with $h^K$ defined by~\eqref{eq:h:K},
  $$
  \delta_K := \min\left\{ \min_{\bm i \in [n]^d} h^K_{\bm i}, \alpha_K \right\} \qquad \text{and} \qquad \Delta_K := \max\left\{ \max_{\bm i \in [n]^d} h^K_{\bm i}, \alpha_K \right\}.
  $$
\end{lem}

When a discrete $I$-projection linear problem is solved via iterated $I$-projections as in Algorithm~\ref{algo:iterated:Iproj} with some of the intermediate $I$-projections computed via scalings whenever feasible, and otherwise via Algorithm~\ref{algo:GIS}, \citet{LinSte23} show in their Theorem~3.1 that one can set the maximum number of iterations $M'$ in Algorithm~\ref{algo:GIS} to be 1 and get a global convergent procedure. In that respect, Algorithm~\ref{algo:iterated:Iproj} with some of the intermediate $I$-projections carried out via generalized iterative scaling with $M' = 1$ is related to the RBI-SMART algorithm studied in \cite{Byr98}. \citet{LinSte23} also mention that this is likely to be computationally faster than setting $M'$ larger than 1. From preliminary numerical experiments (whose final version is reported in Section~\ref{sec:num} and in the supplement \cite{KojMar25}), we can indeed confirm that setting $M' = 1$ in this context seems a better choice in terms of execution time.

We end this subsection by formally stating that Algorithm~\ref{algo:iterated:Iproj} with its Line~\ref{line:Iproj:Gc} based on generalized iterative scaling with $M'=1$ is theoretically justified for approximately solving the generalized minimum information checkerboard copula problem formulated as in~\eqref{eq:GMICC:p} when Condition~\ref{cond:included} holds. While the statement of the following result is somewhat unwieldy for notational reasons, it is merely a consequence of Lemma~\ref{lem:Gc:GIS} and Theorem~3.1 of \cite{LinSte23} as can be seen from its proof given in Appendix~\ref{proofs:solving}. Recall that $N = |\Jc'| + |\Kc|$ and that the $\Fc_J$ in~\eqref{eq:Fc:J} and the $\Lc_K$ in~\eqref{eq:Lc:K} were arbitrarily renamed as $\Ec_1,\dots,\Ec_{|\Jc'|}$ and $\Ec_{|\Jc'|+1},\dots,\Ec_{N}$, respectively. Furthermore, for any $m \in [|\Jc'|]$, let $J_m$ be the set $J \in \Jc'$ corresponding to $\Ec_m$ and, for any $m \in \{|\Jc'|+1,\dots,N\}$, let $K_m$ be the set $K \in \Kc$ corresponding to~$\Ec_m$.

\begin{cor}
  \label{cor:algo:GIS}
  Assume that Condition~\ref{cond:included} holds. Then, the $I$-projection $q$ of $r$ on $\Ec$ in~\eqref{eq:Ec} exists and is unique. Furthermore, let $q^{[0]} := r$ and, for any $m \geq 1$ such that $\change{((m-1) \bmod N) + 1} \in [|\Jc'|]$, let $q^{[m]}$ be defined by
\begin{equation*}
  q^{[m]}_{\bm i} :=
  \begin{cases}
    q^{[m -1]}_{\bm i}  \frac{s^J_{\bm i_J}}{q^{[m-1],(J)}_{\bm i_J}}, &\qquad \text{if } \bm i \in \supp(q^{[m-1]}), \\
    0, &\qquad \text{otherwise},
  \end{cases}
\end{equation*}
where $J = J_{\change{((m-1) \bmod N) + 1}}$ and, for any $m \geq 1$ such that $\change{((m-1) \bmod N) + 1} \in \{|\Jc'|+1,\dots,N\}$, let $q^{[m]}$ be defined by
\begin{equation}
  \label{eq:qm:GIS}
  q^{[m]}_{\bm i} := q^{[m-1]}_{\bm i} \left( \frac{\bar a_K}{\sum_{\bm i' \in [n]^d} q^{[m-1]}_{\bm i'} \bar h^K_{\bm i'}} \right)^{\bar h^K_{\bm i}} \left( \frac{1 - \bar a_K}{\sum_{\bm i' \in [n]^d} q^{[m-1]}_{\bm i'} (1 - \bar h^K_{\bm i'})} \right)^{1 - \bar h^K_{\bm i}}, \qquad \bm i \in [n]^d,
\end{equation}
with the conventions that $0/0 := 0$ and $0^0 := 1$, where $K = K_{\change{((m-1) \bmod N) + 1}}$, and the array $\bar h^K$ and the real $\bar a_K$ are defined as in Lemma~\ref{lem:Gc:GIS}. Then, as $m \to \infty$, $q^{[m]}$ converges to $q$.
\end{cor}

\subsection{$I$-projections on the $\Lc_K$ using a possibly new result}
\label{sec:operational:Gc:new}

We shall now provide a possibly new result that can be used as an alternative to generalized iterative scaling to implement Line~\ref{line:Iproj:Gc} of Algorithm~\ref{algo:iterated:Iproj}. As shall be explained in more detail in Section~\ref{sec:num} (see also the supplement \cite{KojMar25}), in our experiments, the resulting version of Algorithm~\ref{algo:iterated:Iproj} was found to be substantially faster than its version based on generalized iterative scaling with $M' = 1$.

The following proposition, proven in Appendix~\ref{proofs:solving}, can in certain cases be used to compute an $I$-projection on a linear set $\Ec''$ defined from only one expectation constraint as is the case for the sets $\Lc_K$ in~\eqref{eq:Lc:K}.

\begin{prop}
  \label{prop:gen:h}
  Let $q^\dagger \in \Pc_{d,n}$, let $h \in \Ac_{d,n}$ and let $\Lambda$ be the continuous function from $\R$ to $\R$ defined by
  \begin{equation}\label{eq:Lambda}
  \Lambda(\lambda) := \frac{\sum_{\bm i \in [n]^d} h_{\bm i} \, q^\dagger_{\bm i} \exp(\lambda h_{\bm i})}{\sum_{\bm i \in [n]^d} q^\dagger_{\bm i}  \exp(\lambda h_{\bm i})}, \qquad \lambda \in \R.
  \end{equation}
  \begin{enumerate}
  \item[(i)] Assume that $h$ is not constant on $\supp(q^\dagger)$. Then $\Lambda$ is strictly increasing.
  \item[(ii)] Let $a \in \R$ and assume furthermore that $a \in \mathrm{ran}(\Lambda)$. Then, the $I$-projection $q^\star$ of $q^\dagger$ on the set $\Ec'' := \left\{p \in \Pc_{d,n} : \sum_{\bm i \in [n]^d} p_{\bm i} h_{\bm i} = a \right\}$ exists, is unique and is given by
  \begin{equation}
    \label{eq:Iproj:gen}
    q^\star_{\bm i} = \frac{q^\dagger_{\bm i} \exp [ \Lambda^{-1}(a) h_{\bm i} ]}{\sum_{\bm i \in [n]^d} q^\dagger_{\bm i} \exp [ \Lambda^{-1}(a) h_{\bm i}]}, \qquad \bm i \in [n]^d.
  \end{equation}
  \end{enumerate}
\end{prop}


Note that, to carry out the $I$-projection defined via~\eqref{eq:Iproj:gen}, one needs to be able to compute $\Lambda^{-1}$. This can be done in practice by numerical root finding.

The following short result, proven in Appendix~\ref{proofs:solving}, shows that the assumption that $a \in \mathrm{ran}(\Lambda)$ in Proposition~\ref{prop:gen:h} is implied by a simpler condition.

\begin{lem}
  \label{lem:alpha:ran:Lambda}
Let $q^\dagger \in \Pc_{d,n}$, $h \in \Ac_{d,n}$,  $a \in \R$ and assume that there exists $p \in \Ec'' := \left\{p \in \Pc_{d,n} : \sum_{\bm i \in [n]^d} p_{\bm i} h_{\bm i} = a \right\}$ such that $\supp(p) = \supp(q^\dagger)$. Then $a \in \mathrm{ran}(\Lambda)$, where $\Lambda$ is defined in~\eqref{eq:Lambda}.
\end{lem}

We finally state a result showing that the use of Algorithm~\ref{algo:iterated:Iproj} with its Line~\ref{line:Iproj:Gc} based on Proposition~\ref{prop:gen:h} is theoretically justified for approximately solving the generalized minimum information checkerboard copula problem formulated as in~\eqref{eq:GMICC:p} under the following strengthening of Condition~\ref{cond:included}:

\begin{cond}
  \label{cond:equal}
  There exists $p \in \Ec$ in~\eqref{eq:Ec} such that $\supp(p) = \supp(r)$.
\end{cond}

It involves the same notation as the one defined above Corollary~\ref{cor:algo:GIS}. Its proof, given in Appendix~\ref{proofs:solving}, consists of combining Theorem 3.2 of \cite{Csi75} with Corollary~\ref{cor:IPFP:any:marg}, Lemma~\ref{lem:alpha:ran:Lambda} and Proposition~\ref{prop:gen:h}.

\begin{prop}
  \label{prop:algo:gen}
  Assume that Condition~\ref{cond:equal} holds. Then, the $I$-projection $q$ of $r$ on $\Ec$ in~\eqref{eq:Ec} exists and is unique. Assume also that, for any $K \in \Kc$, $h^K$ in~\eqref{eq:h:K} is non constant on $\supp(r)$. Furthermore, let $q^{[0]} := r$ and, for any $m \geq 1$ such that $\change{((m-1) \bmod N) + 1} \in [|\Jc'|]$, let $q^{[m]}$ be defined by
\begin{equation}
  \label{eq:qm:scaling}
  q^{[m]}_{\bm i} :=
  \begin{cases}
    q^{[m -1]}_{\bm i}  \frac{s^J_{\bm i_J}}{q^{[m-1],(J)}_{\bm i_J}}, &\qquad \text{if } \bm i \in \supp(q^{[m-1]}), \\
    0, &\qquad \text{otherwise},
  \end{cases}
\end{equation}
where $J = J_{\change{((m-1) \bmod N) + 1}}$ and, for any $m \geq 1$ such that $\change{((m-1) \bmod N) + 1} \in \{|\Jc'|+1,\dots,N\}$, let $q^{[m]}$ be defined by
\begin{equation}
  \label{eq:qm:gen}
  q^{[m]}_{\bm i} := \frac{q^{[m - 1]}_{\bm i} \exp [ \Lambda^{-1}(a) h_{\bm i} ]}{\sum_{\bm i \in [n]^d} q^{[m - 1]}_{\bm i} \exp [ \Lambda^{-1}(a) h_{\bm i}]}, \qquad \bm i \in [n]^d,
\end{equation}
where $\Lambda$ is defined in~\eqref{eq:Lambda}, $a := \alpha_K$ and $h := h^K$ with $K = K_{\change{((m-1) \bmod N) + 1}}$. Then, as $m \to \infty$, $q^{[m]}$ converges to $q$.
\end{prop}

\subsection{Interpreting the behavior of the generalized iterative scaling version of Algorithm~\ref{algo:iterated:Iproj}}
\label{sec:interpret}

In practice, when using one of the two studied versions of Algorithm~\ref{algo:iterated:Iproj} on a given instance of the generalized minimum information checkerboard copula problem \eqref{eq:GMICC:p}, one will not necessarily know in advance whether the conditions for convergence presented in Corollary~\ref{cor:algo:GIS} or Proposition~\ref{prop:algo:gen} are met. In this section, we focus on the generalized iterative scaling version of Algorithm~\ref{algo:iterated:Iproj} and propose an empirical strategy to assess whether Condition~\ref{cond:included} is likely to be satisfied, based on a simple result. As we shall see below, an adaption of the latter will allow us to obtain insight on marginal compatibility problems of interest.

Consider the following two errors which can be computed for every $m$ when executing Algorithm~\ref{algo:iterated:Iproj}:
\begin{equation}
  \label{eq:max:err}
  \Err_{\Jc'}^{[m]} := \max_{J \in \Jc'} \max_{\bm i_J \in [n]^{|J|}} \left| q^{[m],(J)}_{\bm i_J} - s^J_{\bm i_J} \right|
  \qquad \text{and, if }\Kc \neq \emptyset, \qquad 
  \Err_{\Kc}^{[m]} := \max_{K \in \Kc} \left| \sum_{\bm i \in [n]^d} q^{[m]}_{\bm i} h^K_{\bm i} - \alpha_K \right|.
\end{equation}
These diagnostics provide a numerical assessment of how the iterate $q^{[m]}$ satisfies the marginal constraints defining $\Ec$ relative to $\Jc'$ and $\Kc$, respectively. The following result is proven in Appendix~\ref{proofs:solving}.

\begin{prop}
  \label{prop:max:err}
  Assume that the generalized iterative scaling version of Algorithm~\ref{algo:iterated:Iproj} is executed with $\eps = 0$ and $M = \infty$. Then the following statements are equivalent: 
  \begin{itemize}
  \item[(i)] Condition~\ref{cond:included} holds.
  \item[(ii)] The $I$-projection $q$ of $r$ on $\Ec$ exists, is unique and $q = \lim_{m \to \infty} q^{[m]}$.
  \item[(iii)] $q' = \lim_{m \to \infty} q^{[m]}$ exists with $\supp(q') \subset \supp(r)$, $\lim_{m \to \infty} \Err_{\Jc'}^{[m]} = 0$ and $\lim_{m \to \infty} \Err_{\Kc}^{[m]} = 0$.
  \end{itemize}
\end{prop}

Proposition~\ref{prop:max:err} suggests the following \change{sequential} empirical strategy based on the generalized iterative scaling version of Algorithm~\ref{algo:iterated:Iproj} to empirically assess whether Condition~\ref{cond:included} holds (and thus whether the generalized iterative scaling version of Algorithm~\ref{algo:iterated:Iproj} might have converged to the $I$-projection of $r$ on $\Ec$):

\begin{enumerate}
\item Assess empirically whether $q' = \lim_{m \to \infty} q^{[m]}$ seems to exist by examining the evolution of the maximum absolute error in Line~\ref{line:num:conv} of Algorithm~\ref{algo:iterated:Iproj} against the iteration number. If there is evidence against it, stop and conclude that Condition~\ref{cond:included} is unlikely to hold. 

\item \change{If not,} assess whether $\supp(q') \subset \supp(r)$ by assessing whether $\supp(q^{[m']}) \subset \supp(r)$, where $m'$ is the maximum iteration number reached in Algorithm~\ref{algo:iterated:Iproj}. If there is evidence against it, stop and conclude that Condition~\ref{cond:included} is unlikely to hold. 

\item \change{If not,} assess whether $\lim_{m \to \infty} \Err_{\Jc'}^{[m]} = 0$ and $\lim_{m \to \infty} \Err_{\Kc}^{[m]} = 0$ by examining the evolution of $\Err_{\Jc'}^{[m]}$ and $\Err_{\Kc}^{[m]}$ against the iteration number. If there is evidence against it, stop and conclude that Condition~\ref{cond:included} is unlikely to hold.  

\item \change{If not,} conclude that Condition~\ref{cond:included} is likely to hold and thus that $q^{[m']}$ is likely to be an approximation of $q$, the $I$-projection of $r$ on $\Ec$. 
\end{enumerate}

The following result follows immediately from Proposition~\ref{prop:max:err} and the fact that Condition~\ref{cond:included} is equivalent to $\Ec \neq \emptyset$ if $\supp(r) = [n]^d$ (since in that case $\supp(p) \subset \supp(r)$ for all $p \in \Pc_{d,n}$).

\begin{cor}
  \label{cor:max:err}
  Assume that $\supp(r) = [n]^d$ and that the generalized iterative scaling version of Algorithm~\ref{algo:iterated:Iproj} is executed with $\eps = 0$ and $M = \infty$. Then the following statements are equivalent:
  \begin{itemize}
  \item[(i)] $\Ec \neq \emptyset$.
  \item[(ii)] The $I$-projection $q$ of $r$ on $\Ec$ exists, is unique and $q = \lim_{m \to \infty} q^{[m]}$.
  \item[(iii)] $q' = \lim_{m \to \infty} q^{[m]}$ exists, $\lim_{m \to \infty} \Err_{\Jc'}^{[m]} = 0$ and $\lim_{m \to \infty} \Err_{\Kc}^{[m]} = 0$.
  \end{itemize}
\end{cor}

As a consequence of Corollary~\ref{cor:max:err}, when $\supp(r) = [n]^d$, one can adapt the previous empirical strategy to attempt to assess whether $\Ec \neq \emptyset$ (that is, the consistency of the constraints) by merely removing Step~2. This is illustrated in the supplement \cite{KojMar25} for marginal compatibility problems.

\begin{remark}
  \label{rem:nonconv}
  A practical difficulty with the above empirical strategy arises from the fact that distinguishing between nonconvergence and slow convergence may be impossible in practice. For a theoretical perspective on such issues in the case of the bivariate iterated proportional fitting procedure, see, e.g., \cite{BroLeu18}. \qed
\end{remark}

\begin{remark}
  \label{rem:max:err}
The reason why the two errors in~\eqref{eq:max:err} are not merged into one global error is practical: as our experiments show, they are typically on different scales, so that a practitioner may not want to use the same threshold to decide whether they are sufficiently small or not. Of course, the errors $\Err_{\Kc}^{[m]}$ could also be further split according to the cardinality of $K$ and the nature of the underlying expectation. Furthermore, the errors in~\eqref{eq:max:err} could be used to define alternative stopping criteria in Algorithm~\ref{algo:iterated:Iproj} but we do not pursue this here. Finally, note that, to save run time when executing Algorithm~\ref{algo:iterated:Iproj}, we suggest to compute the diagnostics in~\eqref{eq:max:err} only for a reasonably limited number of values of $m$ (for instance for a subset of those satisfying $\change{((m-1) \bmod N) + 1 = N}$). \qed
\end{remark}


\section{Selected numerical experiments}
\label{sec:num}

To illustrate the results of the previous section, we applied them to attempt to approximately solve the generalized minimum information checkerboard copula problem in~\eqref{eq:GMICC} when:
\begin{enumerate}
\item $\check R = \check U_d = U_d$, where $U_d$ is probability measure of the uniform distribution on $[0,1]^d$,

\item if $\Kc \neq \emptyset$, all the $K \in \Kc$ are of cardinality~2 and all the functions $G_K$ are equal to the function $\rho$ in~\eqref{eq:rho} related to Spearman's rho.

\end{enumerate}
Under the choice made in Point~1, solving~\eqref{eq:GMICC} can be interpreted as applying a maximum entropy principle as already mentioned in the introduction. Concerning Point~2, from~\eqref{eq:g:rho} and~\eqref{eq:h:K}, we immediately obtain that, for any $K = \{\ell_1,\ell_2\} \in \Kc$ and $\bm i \in [n]^d$,
\begin{equation}
  \label{eq:h:rho}
  \begin{split}
    h^{\{\ell_1,\ell_2\}}_{\bm i} &= n^2 \int_{\frac{i_{\ell_1}-1}{n}}^{\frac{i_{\ell_1}}{n}} \int_{\frac{i_{\ell_2}-1}{n}}^{\frac{i_{\ell_2}}{n}} 12 \left(v_1 - \frac{1}{2} \right) \left(v_2 - \frac{1}{2}\right) \dd v_1 \dd v_2 \\
                                  &= 12 \left(\frac{i_{\ell_1}-1}{n} + \frac{1}{2n} - \frac{1}{2} \right) \left( \frac{i_{\ell_2}-1}{n} + \frac{1}{2n} - \frac{1}{2} \right) \\
                                  &= g_\rho \left(\frac{i_{\ell_1}-1}{n} + \frac{1}{2n}, \frac{i_{\ell_2}-1}{n} + \frac{1}{2n} \right).
  \end{split}
\end{equation}
Note that, as verified for instance in \cite[Section~4]{PiaHowBor12}, for any $\check P \in \check\Cc_{n}([0,1]^2)$, $\rho(\check P) \in [-1 + 1/n^2, 1 - 1/n^2]$. Hence, for a given value of the discretization parameter~$n$, all the $\alpha_K$ in~\eqref{eq:GMICC} need to be taken in $[-1 + 1/n^2, 1 - 1/n^2]$. Of course, even if $\Jc = \emptyset$ (recall for instance from~\eqref{eq:Jc'} that $\Jc = \Jc' \setminus \{\{1\},\dots,\{d\}\}$), this is not sufficient to guarantee that the constraints are consistent \citep[see, e.g.,][Section~6]{PiaHowBor12}.

From Section~\ref{sec:sol}, the resulting generalized minimum information checkerboard copula problem can then be reformulated in terms of probability arrays as $\min_{p \in \Ec} I(p \| u)$, where $\Ec$ is defined in~\eqref{eq:Ec} and $u \in \Cc_{d,n}$ is the skeleton of $\check U_d$. Note that since $\supp(u) = [n]^d$, Condition~\ref{cond:included} is simply equivalent to $\Ec \neq \emptyset$. From Proposition~\ref{prop:unicity:p} and Section~\ref{sec:solving}, to approximately find the unique solution $q$ of the aforementioned problem when $\Ec \neq \emptyset$, we have the following two possibilities: if there exists $p \in \Ec$ such that $\supp(p) \subsetneq [n]^d$, we can use Algorithm~\ref{algo:iterated:Iproj} with its Line~\ref{line:Iproj:Gc} implemented using generalized iterative scaling with $M' = 1$ as considered in Section~\ref{sec:operational:Gc:GIS}; if Condition~\ref{cond:equal} with $r = u$ holds, that is, if there exists $p \in \Ec$ such that $\supp(p) = [n]^d$, in addition to the previous approach, we can use Algorithm~\ref{algo:iterated:Iproj} with its Line~\ref{line:Iproj:Gc} implemented using~\eqref{eq:qm:gen} as considered in Section~\ref{sec:operational:Gc:new}. For ease of reference, we shall refer to the former (resp.\ latter) version of Algorithm~\ref{algo:iterated:Iproj} as Procedure~I (resp.~II). Note that the theoretical validity of Procedure II under Condition~\ref{cond:equal} with $r = u$ follows from Proposition~\ref{prop:algo:gen} combined with the fact that, for any $\{\ell_1,\ell_2\} \in \Kc$, $h^{\{\ell_1,\ell_2\}}$ in~\eqref{eq:h:rho} is not constant on $[n]^2$ as soon as $n \geq 2$. As already mentioned, to compute $\Lambda^{-1}$ in~\eqref{eq:qm:gen} in practice, we use numerical root finding.

As already discussed in Section~\ref{sec:interpret}, in practice, given a value of the discretization parameter $n$ and an instance of the considered version of the generalized minimum information checkerboard copula problem, we will not necessarily know in advance whether $\Ec \neq \emptyset$ or if there exists $p \in \Ec$ such that $\supp(p) = [n]^d$. To help interpret the behavior of Procedure~I, we therefore rely on the empirical strategy proposed in Section~\ref{sec:interpret} once adapted to the setting $\supp(r) = [n]^d$ corresponding to Corollary~\ref{cor:max:err}. More precisely, when executing Procedure~I, we may conclude that the underlying optimization problem admits a solution (equivalently, that $\Ec \neq \emptyset$), and that the last array returned by Algorithm~\ref{algo:iterated:Iproj} provides a reasonable approximation of that solution, if and only if the iterates $q^{[m]}$ appear to stabilize as $m$ increases and if the quantities $\Err_{\Jc'}^{[m]}$ and $\Err_{\Kc}^{[m]}$ defined in~\eqref{eq:max:err} become sufficiently small for large $m$. This is particularly important in relation to marginal compatibility problems as illustrated in the experiments reported in the supplement \cite{KojMar25}. 

The quantities $\Err_{\Jc'}^{[m]}$ and $\Err_{\Kc}^{[m]}$ defined in~\eqref{eq:max:err}, and the maximum absolute error in Line~\ref{line:num:conv} of Algorithm~\ref{algo:iterated:Iproj} actually provide fundamental diagnostics for both procedures. When analyzing them, we need to keep in mind the warning given in Remark~\ref{rem:nonconv}: distinguishing between nonconvergence and slow convergence may be difficult in practice.

Our numerical experiments consisted of considering several instances of the generalized minimum information checkerboard copula problem with $\check R = U_d$ and Spearman's rho constraints, and attempting to solve them for various choices of the discretization parameter~$n$. Clearly, possible values for $n$ depend on the value of $d$ as Algorithm~\ref{algo:iterated:Iproj} manipulates probability arrays of $n^d$ elements. To allow us to consider relatively large values of $n$ (so that, roughly speaking, the checkerboard problem is a good approximation of the corresponding generalized minimum information copula problem -- see \eqref{eq:approx}), we restricted ourselves to the low-dimensional case, that is, $d \in \{2,3,4\}$. For the sake of brevity, we only (partially) report hereafter our results for $d \in \{2,4\}$. Additional figures as well as the results for $d = 3$, related to sanity checks and marginal compatibility problems, among others, are available in the supplement \cite{KojMar25}.

\subsection{Bivariate experiment with a Spearman's rho constraint}

We begin with a simple bivariate example with $\Kc = \{\{1,2\}\}$ and $\alpha_{\{1,2\}} = 0.8$ (note that then $\Jc = \emptyset$, that is, $\Jc' = \{\{1\}, \{2\}\}$). This setting is similar to some of those considered in \cite{SukSei25a}. We first consider $n = 30$. The parameters $\eps$ and $M$ of Algorithm~\ref{algo:iterated:Iproj} are fixed to $10^{-14}$ and $10 \, 000$, respectively. The following output provides a summary of the execution of our \textsf{R} implementation when Line~\ref{line:Iproj:Gc} of Algorithm~\ref{algo:iterated:Iproj} is implemented using generalized iterative scaling with $M' = 1$ (Procedure I):
\begin{small}
  \verbatiminput{output/d2n30GIS.txt}
\end{small}
When Line~\ref{line:Iproj:Gc} is implemented using~\eqref{eq:qm:gen} (Procedure II), we obtain:
\begin{small}
  \verbatiminput{output/d2n30.txt}
\end{small}
As one can see, for both versions of Algorithm~\ref{algo:iterated:Iproj}, the convergence criterion is satisfied although approximately 10 times more iterations were necessary for Procedure I. The evolution of the maximum absolute error (in Line~\ref{line:num:conv} of Algorithm~\ref{algo:iterated:Iproj}) and of the diagnostics in~\eqref{eq:max:err} against the iteration number are given in Figure~\ref{fig:2d:n30:errs} of the supplement \cite{KojMar25}. From the last column of plots, one can see that the numerical satisfaction of the Spearman's rho constraint keeps on improving with the iteration number for Procedure I while for Procedure II the error stagnates at $10^{-5}$ approximately. We believe that this is merely a consequence of the (default) precision with which we run the numerical root finding necessary to execute Procedure~II.

From the output of Procedure II for instance, we see that the maximum absolute error (in Line~\ref{line:num:conv} in Algorithm~\ref{algo:iterated:Iproj}) dropped below $\eps = 10^{-14}$ after 144 iterations. Since $|\Jc'|+|\Kc|=3$, this error is the maximum absolute error between the probability arrays $q^{[3 \times 144]}$ and $q^{[3 \times 143]}$. Here, $q^{[3 \times 144]}$ is the final approximation of the $I$-projection of $u$ on the set $\Ec$ in~\eqref{eq:Ec}. We also see that the maximum absolute error between the first (resp.\ second) margin of $q^{[3 \times 144]}$ and $u_1$ is smaller than $10^{-13}$ and that $\rho(\check Q^{[3 \times 144]}) \approx 0.799992$. In other words, the output confirms ``numerically'' that the probability array $q^{[3 \times 144]}$ has uniform margins (i.e., that it is a copula array) and that the corresponding checkerboard copula has the desired Spearman's rho. The total execution time with our \textsf{R} implementation (on a standard laptop computer) was approximately 0.003~min. Such a timing has no absolute meaning but can be used to compare this example with the forthcoming examples in terms of computational cost.  We next generate one realization of a random sample of size $1000$ from a discrete random vector $(V_1,V_2)$ with support
\begin{equation}
  \label{eq:Xc}
  \Xc := \left\{ \left(\frac{i_1-1}{n} + \frac{1}{2n}, \frac{i_2-1}{n} + \frac{1}{2n} \right) : (i_1,i_2) \in [n]^2 \right\}
\end{equation}
and probability array $q^{[3 \times 144]}$. Note that the points in $\Xc$ are the centers of the $B_{\bm i}$ in~\eqref{eq:B:i} when $d=2$. The resulting scatterplot is represented in the left panel of Figure~\ref{fig:2d:n30}. Following Remark~\ref{rem:check}, one may argue that it is more natural to sample from $\check Q^{[3 \times 144]}$: the resulting scatterplot is given in the right panel of Figure~\ref{fig:2d:n30}. 

\begin{figure}[t!]
  \begin{center}
    \includegraphics*[width=0.3\linewidth]{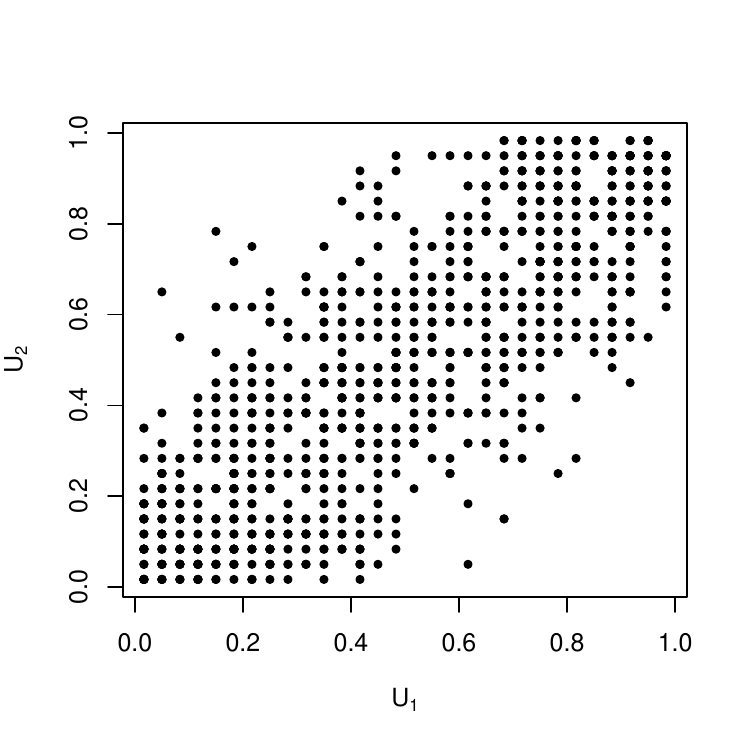}
    \includegraphics*[width=0.3\linewidth]{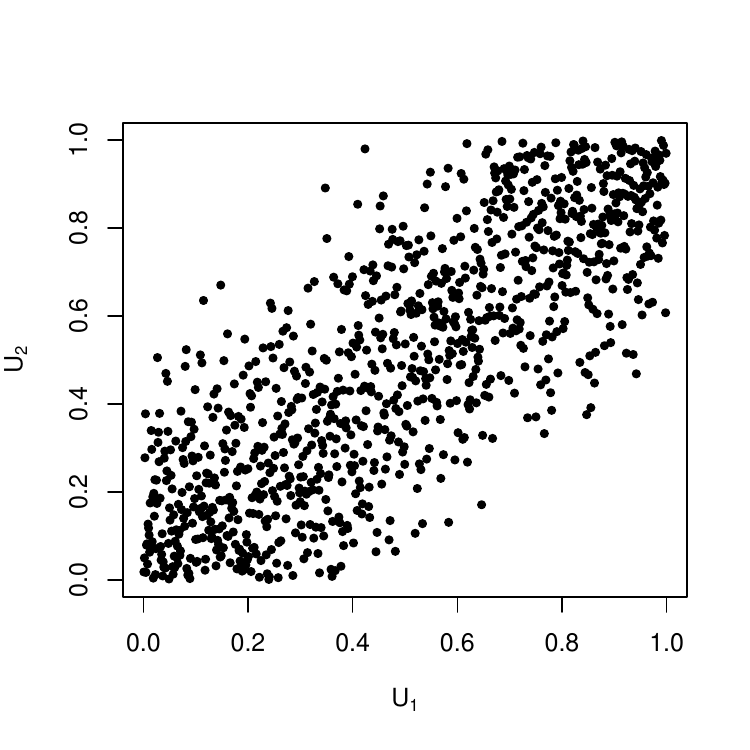}
    \caption{\label{fig:2d:n30} Results obtained with Procedure II when $n = 30$. Left: one realization of a random sample of size $1000$ from a discrete random vector $(V_1,V_2)$ with support $\Xc$ in~\eqref{eq:Xc} and probability array $q^{[3 \times 144]}$. Right: one realization of a random sample of size $1000$ from $\check Q^{[3 \times 144]}$.}
\end{center}
\end{figure}

We next rerun the previous experiment with $n=300$. The execution outputs, given below, are very similar with the exception that the execution times are now approximately 100 times greater:
\begin{small}
  \verbatiminput{output/d2n300GIS.txt}
\end{small}
\begin{small}
  \verbatiminput{output/d2n300.txt}
\end{small}
We also see that the generalized iterative scaling version of Algorithm~1, that is, Procedure I (first output) leads (again) to a better numerical satisfaction of the Spearman's rho constraint. Realizations of random samples from the discrete model obtained from Procedure II and the associated checkerboard copula are represented in Figure~\ref{fig:2d:n300} of the supplement \cite{KojMar25}. Unsurprisingly, since $n=300$, there are no visually noticeable differences between the empirical distributions in the left and right panels.

\subsection{Trivariate experiments}

The results of 7 additional trivariate experiments related to sanity checks and marginal compatibility problems, among others, are reported in the supplement \cite{KojMar25}.

\subsection{A 4-dimensional experiment}
\label{sec:num:d4}

We end this section with a 4-dimensional experiment for which
$$
\Jc = \{\{1,2\}, \{1,3\}\} \qquad \text{and} \qquad \Kc = \{ \{1,4\}, \{2,3\}, \{2,4\}, \{3,4\} \}.
$$
We set $\eps = 10^{-10}$, $M = 10 \, 000$ and $n=100$. Furthermore, we take the skeletons of $\check S^{\{1,2\}}$ and $\check S^{\{1,3\}}$ in~\eqref{eq:GMICC} to correspond to discretizations of the Clayton copula with parameter~3 and the Gumbel--Hougaard copula with parameter~2, respectively. In addition, we set $\alpha_{\{1,4\}} = 0.7$, $\alpha_{\{2,3\}} = 0.3$, $\alpha_{\{2,4\}} = 0.4$ and $\alpha_{\{3,4\}} = 0.7$. We only execute Procedure II as we expect it to be substantially faster than Procedure I for this setting. The resulting execution output is:
\begin{small}
  \verbatiminput{output/d4.txt}
\end{small}
As one can see, the numerical convergence criterion is satisfied and the resulting probability array appears to satisfy the imposed constraints rather well. The large execution time (1 hour approximately) for only 173 iterations can be explained by the fact that the procedure manipulated probability arrays with $100^4$ elements. A realization of a random sample from the returned 4-dimensional checkerboard copula is displayed in Figure~\ref{fig:4d} while diagnostic (error) plots can be found in Figure~\ref{fig:4d:err} of the supplement \cite{KojMar25}.

\begin{figure}[t!]
\begin{center}
  \includegraphics*[width=0.7\linewidth]{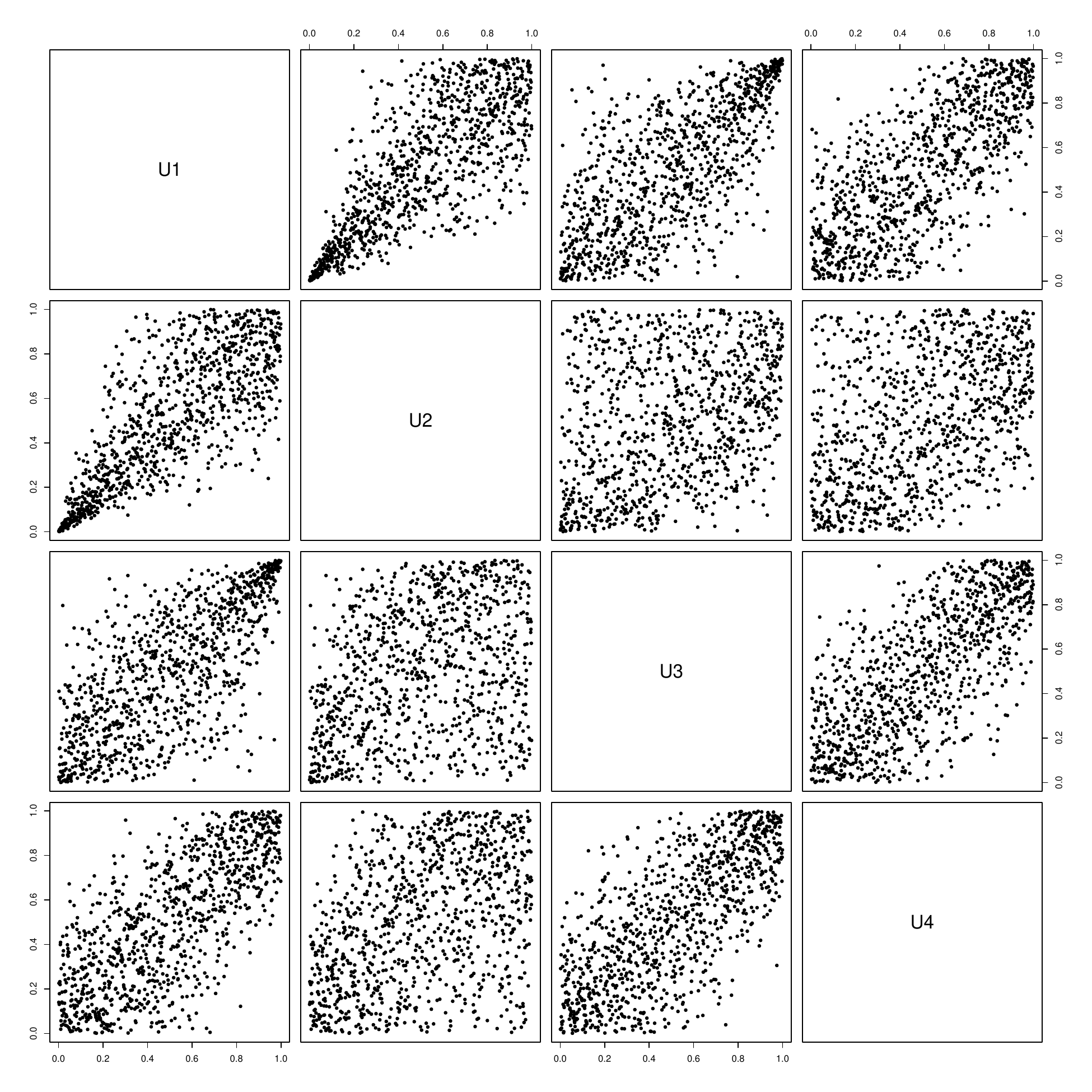}
  \caption{\label{fig:4d} Scatterplot matrix of a realization of a random sample of size $1000$ from the 4-dimensional checkerboard copula returned in the 4-dimensional experiment.}
\end{center}
\end{figure}

\section{Concluding remarks}
\label{sec:conclusion}

We conclude this work by three brief remarks on possible extensions of this work: 
\begin{itemize}


\item The algorithmic approach derived in Section~\ref{sec:solving} for attempting to solve the generalized minimum information checkerboard copula problem works on the skeletons of the underlying checkerboard probability measures. Such probability arrays turn out to be also at the heart of the so-called discrete copula approach initially put forward in \cite{Gee20}, and further studied in \cite{KojMar24,GeeKojMar25}. As a consequence, in future work, the proposed iterative $I$-projection procedure could be directly used to solve what could be called the generalized minimum information discrete copula problem. The only difference would be that the latter would involve probability arrays whose $d$ dimensions have possibly different sizes.


\item \change{As mentioned in Section~\ref{sec:num}, the main practical issue related to the use of the proposed algorithmic procedure in its general form is that it manipulates arrays with $n^d$ elements. With our non optimal implementation, we found it possible to set $n = 100$ when $d=4$ (see Section~\ref{sec:num:d4}). Note that the same memory storage cost would be incurred by setting $n = 10$ when $d = 8$. In future research, we will investigate under which conditions memory saving implementations of the proposed algorithmic procedure are possible.}

\item \change{Another challenge related to the use of the proposed algorithmic procedure when $d$ increases is that the number of possible constraints grows exponentially with $d$. The larger the number of constraints, the higher the probability of inconsistency. A possible strategy when the execution of Algorithm~\ref{algo:iterated:Iproj} suggests that the constraints are inconsistent (see Section~\ref{sec:interpret}) would be to explore the space of subproblems (problems in which one or more constraints related to $J \in \Jc$ or $K \in \Kc$ are removed) to try to identify which constraints may cause the possible global inconsistency.}

\end{itemize}


\setcounter{section}{0}
\renewcommand{\thesection}{\Alph{section}}

\section{Proofs of the results of Section~\ref{sec:MIC}}
\label{proofs:MIC}

\begin{proof}[\bf Proof of Proposition~\ref{prop:unicity}]
For any $\ell \in [d]$, let $S^{\{\ell\}} \in \Mc([0,1])$ be equal to $U_1$, the probability measure of the univariate standard uniform distribution, and let $F_{\{\ell\}} :=  \{P \in \Mc([0,1]^d) : P^{(\{\ell\})} = S^{\{\ell\}} \}$. This extends the definition in~\eqref{eq:F:J} to all $J \in \Jc'$, where $\Jc'$ is defined in~\eqref{eq:Jc'}. Then, using the fact that $\Cc([0,1]^d) = \bigcap_{\ell \in [d]} F_{\{\ell\}}$, $E$ in~\eqref{eq:E} can be rewritten as $E = \bigcap_{J \in \Jc'} F_J \cap \bigcap_{K \in \Kc} L_K$. Next, it is easy to verify each $F_J$, $J \in \Jc'$, and each $L_K$ in~\eqref{eq:L:K} is convex, and that this implies that $E$ is convex. Furthermore, from Lemmas~\ref{lem:F:closed} and~\ref{lem:G:closed} below, each $F_J$ and each $L_K$ is closed in total variation, which then implies that $E$ is closed in total variation. Since $E$ is convex and closed in total variation, and there exists $P \in E$ such that $I(P \| R) < \infty$, we know from Theorem~2.1 of \cite{Csi75} that there exists a unique $Q \in E$ such that $I(Q \| R) = \min_{P \in E} I(P \| R) < \infty$. Hence, $Q \ll R$ from~\eqref{eq:KL}. Furthermore, from the remark following Theorem~2.2 in \cite{Csi75}, $P \ll Q$ for all $P \in E$ such that $I(P \|R) < \infty$.
\end{proof}

\begin{lem}
  \label{lem:F:closed}
  For any $J \in \Jc'$, the set $F_J$ in~\eqref{eq:F:J} is closed in total variation.
\end{lem}

\begin{proof}
  Let $\tau$ denote the total variation metric and fix $J \in \Jc'$. Next, let $(P_m)_{m \in \N} \in F_J$ such that $\lim_{m \to \infty} \tau(P_m,P) = 0$ for some $P \in \Mc([0,1]^d)$. Using the definition of $\tau$ as well as the definitions of $P_m^{(J)}$ and $P^{(J)}$ (see Section~\ref{sec:margins}), it is easy to verify that $\lim_{m \to \infty} \tau(P_m,P) = 0$ implies that $\lim_{m \to \infty} \tau(P_m^{(J)},P^{(J)}) = 0$. Since convergence in total variation implies setwise convergence, we then have that
  \begin{equation}
    \label{eq:setwise:1}
    \lim_{m \to \infty} P_m^{(J)}(B) = P^{(J)}(B), \qquad \forall B \in \Bc_{\R^{|J|}}.
  \end{equation}
  Furthermore, by assumption, $P_m \in F_J$ for all $m \in \N$, that is, $P_m^{(J)} = S^J$ for all $m \in \N$. This implies that
 \begin{equation}
    \label{eq:setwise:2}
    \lim_{m \to \infty} P_m^{(J)}(B) = \lim_{m \to \infty} S^J(B) = S^J(B), \qquad \forall B \in \Bc_{\R^{|J|}}.
  \end{equation}
Hence, by~\eqref{eq:setwise:1} and~\eqref{eq:setwise:2}, $P^{(J)}(B) = S^J(B)$ for all $B \in \Bc_{\R^{|J|}}$, that is, $P \in F_J$.
\end{proof}

\begin{lem}
  \label{lem:G:closed}
  For any $K \in \Kc$, the set $L_K$ in~\eqref{eq:L:K} is closed in total variation.
\end{lem}

\begin{proof}
  Fix $K \in \Kc$ and let $(P_m)_{m \in \N} \in L_K$ such that $\lim_{m \to \infty} \tau(P_m,P) = 0$ for some $P \in \Mc([0,1]^d)$. Since $(P_m)_{m \in \N} \in L_K$, for any $m \in \N$,
  \begin{equation}
    \label{eq:moment}
    \int_{[0,1]^d} g_K(\bm v) \dd P_m(\bm v) = \alpha_K.
  \end{equation}
  Since convergence in total variation implies weak convergence and the function $g_K$ is continuous (and thus bounded) on $[0,1]^d$, $\lim_{m \to \infty} \tau(P_m,P) = 0$ implies that
  $$
  \lim_{m \to \infty} \int_{[0,1]^d} g_K(\bm v) \dd P_m(\bm v) = \int_{[0,1]^d} g_K(\bm v) \dd P(\bm v),
  $$
  so that from~\eqref{eq:moment}, $\int_{[0,1]^d} g_K(\bm v) \dd P(\bm v) = \alpha_K$, and hence from~\eqref{eq:G:K}, $P \in L_K$.
\end{proof}


\section{Proofs of the results of Section~\ref{sec:solving}}
\label{proofs:solving}

The proof of Corollary~\ref{cor:IPFP:any:marg} (given below) is a consequence of the first result mentioned in Section~5.1 of \cite{CsiShi04}. For completeness, we first provide a statement of the latter with our notation along with a proof.

\begin{prop}
  \label{prop:scaling}
  Let $q^\dagger \in \Pc_{d,n}$, let $\{\Bc_1,\dots,\Bc_b\}$  be a partition of $[n]^d$ and let $a_1,\dots,a_b \in [0,1]$ such that $\sum_{k = 1}^b a_k = 1$ and $a_k = 0$ if $\Bc_k \subset [n]^d \setminus \supp(q^\dagger)$. 
  Furthermore, for any $k \in [b]$, let $h_k \in \Ac_{d,n}$ be defined by $h_{k, \bm i} = \1_{\Bc_k}(\bm i)$, $\bm i \in [n]^d$. Then, the $I$-projection $q^\star$ of $q^\dagger$ on
$$
\Ec' = \bigcap_{k \in [b]} \left\{p \in \Pc_{d,n} : \sum_{\bm i \in [n]^d} p_{\bm i} h_{k,\bm i} = a_k \right\} = \bigcap_{k \in [b]} \left\{p \in \Pc_{d,n} : \sum_{\bm i \in \Bc_k} p_{\bm i}  = a_k \right\}
$$
exists, is unique and is given by

  \begin{equation}
    \label{eq:Iproj:scaling}
    q^\star_{\bm i} = \begin{cases}
    q^\dagger_{\bm i} \sum_{k \in [b]} \1_{\Bc_k}(\bm i)\frac{a_k}{\sum_{\bm i' \in \Bc_k} q^\dagger_{\bm i'}}, &\qquad \text{if } \bm i \in \supp(q^\dagger), \\
      0, &\qquad \text{otherwise}.
       \end{cases}
   \end{equation}

  \end{prop}

\begin{proof}[\bf Proof of Proposition~\ref{prop:scaling}]

  Let us first check that $q^\star$ in~\eqref{eq:Iproj:scaling} is a probability array. It is easy to see from~\eqref{eq:Iproj:scaling} and the assumptions on the $a_k$ that $q^\star_{\bm i} \geq 0$ for all $\bm i \in [n]^d$. Furthermore,
  \begin{align*}
    \sum_{\bm i \in [n]^d} q^\star_{\bm i} &= \sum_{\bm i \in \supp(q^\dagger)} q^\star_{\bm i} = \sum_{\bm i \in \supp(q^\dagger)}q^\dagger_{\bm i} \sum_{k \in [b]} \1_{\Bc_k}(\bm i)\frac{a_k}{\sum_{\bm i' \in \Bc_k} q^\dagger_{\bm i'}} \\
                                           &= \sum_{k \in [b]}  \frac{a_k}{\sum_{\bm i' \in \Bc_k} q^\dagger_{\bm i'}} \sum_{\bm i \in \supp(q^\dagger)}q^\dagger_{\bm i} \1_{\Bc_k}(\bm i) = 1.
  \end{align*}

  Let us next check that the probability array $q^\star$ in~\eqref{eq:Iproj:scaling} belongs to $\Ec'$. This is equivalent to verifying that, for any $k \in [b]$, $\sum_{\bm i \in \Bc_k} q^\star_{\bm i} = a_k$. On one hand, if $\Bc_k \subset [n]^d \setminus \supp(q^\dagger)$, $\sum_{\bm i \in \Bc_k} q^\star_{\bm i} = 0$ by~\eqref{eq:Iproj:scaling} and $a_k = 0$ from the assumptions. On the other end, if $\Bc_k \cap \supp(q^\dagger) \neq \emptyset$, $\sum_{\bm i \in \Bc_k} q^\star_{\bm i} = a_k$ from~\eqref{eq:Iproj:scaling}. Hence, $q^\star$ in~\eqref{eq:Iproj:scaling} belongs to $\Ec'$.

  Since $\supp(q^\star) \subset \supp(q^\dagger)$, it follows that there exists $p \in \Ec'$ such that $\supp(p) \subset \supp(q^\dagger)$. Let $\Dc(q^\dagger) = \{p \in \Pc_{d,n} :\supp(p) \subset \supp(q^\dagger)\}$. The fact that the $I$-projection of $q^\dagger$ on $\Ec'$ exists and is unique then follows from Theorem~2.1 of \cite{Csi75} (since $\Ec'$ is convex and closed) and, from~\eqref{eq:KL:discrete}, we have that
  \begin{equation}\label{eq:discr:supp:Iproj}
  \min_{p \in \Ec'} I(p \| q^\dagger) = \min_{p \in \Ec' \cap \Dc(q^\dagger)} I(p \| q^\dagger).
  \end{equation}
 From Jensen's inequality \citep[see also][Lemma~4.1]{CsiShi04}, we know that, for any $p \in \Pc_{d,n} \cap \Dc(q^\dagger)$,
  $$
  \sum_{k \in [b]} \left(\sum_{\bm i \in \Bc_k} p_{\bm i}\right) \log \left(\frac{\sum_{\bm i \in \Bc_k} p_{\bm i}}{\sum_{\bm i \in \Bc_k} q^\dagger_{\bm i}} \right) \leq I(p \| q^\dagger),
  $$
  with the conventions that $0 \log 0 = 0$ and $0 \log(0/0) = 0$. This implies that, for any $p \in \Ec' \cap \Dc(q^\dagger)$,
  $$
  \sum_{k \in [b]} a_k \log \left(\frac{a_k}{\sum_{\bm i \in \Bc_k} q^\dagger_{\bm i}} \right) \leq I(p \| q^\dagger).
  $$
  However,
  \begin{align*}
    I(q^\star \| q^\dagger) &= \sum_{\bm i \in \supp(q^\dagger)} q^\dagger_{\bm i} \sum_{k \in [b]} \1_{\Bc_k}(\bm i)\frac{a_k}{\sum_{\bm i' \in \Bc_k} q^\dagger_{\bm i'}} \log \left(  \sum_{k' \in [b]} \1_{\Bc_{k'}}(\bm i)
                              \frac{a_{k'}}{\sum_{\bm i' \in \Bc_{k'}} q^\dagger_{\bm i'}} \right) \\
                            &= \sum_{k \in [b]}  \frac{a_k}{\sum_{\bm i' \in \Bc_k} q^\dagger_{\bm i'}} \sum_{\bm i \in \supp(q^\dagger)}  \1_{\Bc_k}(\bm i) q^\dagger_{\bm i} \log \left(  \frac{a_k}{\sum_{\bm i' \in \Bc_k} q^\dagger_{\bm i'}} \right) = \sum_{k \in [b]} a_k \log \left(\frac{a_k}{\sum_{\bm i \in \Bc_k} q^\dagger_{\bm i}} \right),
  \end{align*}
  which implies that $I(q^\star \| q^\dagger) = \min_{p \in \Ec' \cap \Dc(q^\dagger)} I(p \| q^\dagger)$ and thus, from~\eqref{eq:discr:supp:Iproj}, that $q^\star$ is indeed the $I$-projection of $q^\dagger$ on $\Ec'$.
\end{proof}


\begin{proof}[\bf Proof of Corollary~\ref{cor:IPFP:any:marg}]
  Consider the partition $\{ \Bc_{\bm i_J^* } \}_{\bm i_J^* \in [n]^{|J|}}$ of $[n]^d$, where $\Bc_{\bm i_J^* } = \{ \bm i \in [n]^d: \bm i_J = \bm i_J^* \}$. Then, since $\bm i \in \Bc_{\bm i_J^*} \iff \bm i_J = \bm i_J^*$,
  $$
  \Fc_J 
  = \bigcap_{\bm i_J^* \in [n]^{|J|}} \left\{p \in \Pc_{d,n} :   p^{(J)}_{\bm i_J^*} = s^J_{\bm i_J^*} \right\} = \bigcap_{\bm i_J^* \in [n]^{|J|}} \left\{p \in \Pc_{d,n} :   \sum_{\bm i \in \Bc_{\bm i_J^* }} p_{\bm i} = s^J_{\bm i_J^*} \right\}.
  $$
  Note that the $s^J_{\bm i_J^*}$, $\bm i_J^* \in [n]^{|J|}$, are nonnegative and sum up to one and, by definition, the assumption that $\supp (s^J) \subset \supp(q^{\dagger, (J)})$ is equivalent to $q^{\dagger, (J)}_{\bm i_J^*} = \sum_{\bm i \in \Bc_{\bm i_J^* }} q^\dagger_{\bm i} = 0$ implies $s^J_{\bm i_J^*} = 0$. Hence, $\Bc_{\bm i_J^* } \subset [n]^d \setminus \supp(q^{\dagger})$ implies $s^J_{\bm i_J^*} = 0$. We can then apply Proposition~\ref{prop:scaling} to obtain that the $I$-projection $q^\star$ of $q^\dagger$ on $\Fc_J$ exists, is unique and is given by
  $$
  q^\star_{\bm i} = q^\dagger_{\bm i} \sum_{\bm i_J^* \in [n]^{|J|}} \1_{\Bc_{\bm i_J^*}}(\bm i)\frac{s^J_{\bm i_J^*}}{\sum_{\bm i' \in \Bc_{\bm i_J^*}} q^\dagger_{\bm i'}} = q^\dagger_{\bm i} \frac{s^J_{\bm i_J}}{\sum_{\bm i' \in \Bc_{\bm i_J}} q^\dagger_{\bm i'}} = q^\dagger_{\bm i} \frac{s^J_{\bm i_J}}{q^{\dagger,(J)}_{\bm i_J}},
  $$
  where we have used that $\bm i \in \Bc_{\bm i_J^*} \iff \bm i_J = \bm i_J^*$.
\end{proof}

\begin{proof}[\bf Proof of Lemma~\ref{lem:Gc:GIS}]
  Fix $K \in \Kc$. Note that, for any $\bm i \in [n]^d$, $\bar h^K_{\bm i} \in [0,1]$, and $\bar a_K \in [0,1]$. Furthermore, from Theorem~\ref{thm:GIS}, some thought reveals that~\eqref{eq:update:GIS:Gc} is the analog of~\eqref{eq:update:GIS} when attempting to $I$-project on $\Lc' \cap \Lc''$, where
  $$
  \Lc' := \left\{p \in \Pc_{d,n} : \sum_{\bm i \in [n]^d} p_{\bm i} \bar h^K_{\bm i} = \bar a_K \right\} \; \text{and} \; \Lc'' := \left\{p \in \Pc_{d,n} : \sum_{\bm i \in [n]^d} p_{\bm i} (1 - \bar h^K_{\bm i}) = 1 - \bar a_K \right\}.
  $$
  But it is easy to see $\Lc' = \Lc''$ so that $\Lc' \cap \Lc'' = \Lc'$. Moreover,
  \begin{align*}
    p \in \Lc' &\iff \sum_{\bm i \in [n]^d} p_{\bm i} \bar h^K_{\bm i} = \bar a_K \iff \sum_{\bm i \in [n]^d} p_{\bm i} \frac{h^K_{\bm i} - \delta_K}{\Delta_K - \delta_K} = \frac{\alpha_K - \delta_K}{\Delta_K - \delta_K} \iff p \in \Lc_K,
  \end{align*}
  so that $\Lc' = \Lc_K$ and the proof is complete.
\end{proof}

\begin{proof}[\bf Proof of Corollary~\ref{cor:algo:GIS}]
  First, from Lemma~\ref{lem:Gc:GIS}, we know that~\eqref{eq:qm:GIS} is a suitable specialization of~\eqref{eq:update:GIS} in Theorem~\ref{thm:GIS} when attempting to use generalized iterative scaling to $I$-project on some $\Lc_K$ in~\eqref{eq:Lc:K}. In other words, the setting of Corollary~\ref{cor:algo:GIS} does correspond to Algorithm~\ref{algo:iterated:Iproj} with its Line~\ref{line:Iproj:Gc} based on generalized iterative scaling with $M'=1$. Next, as already mentioned, Condition~\ref{cond:included} implies that there exists $p \in \Ec$ in~\eqref{eq:Ec} such that $\supp(p) \subset \supp(r)$. Let $\Dc(r) = \{p \in \Pc_{d,n} :\supp(p) \subset \supp(r)\}$. The fact that the $I$-projection $q$ of $r$ on $\Ec$ exists and is unique then follows from Theorem~2.1 of \cite{Csi75} (since $\Ec$ is convex and closed) and, from~\eqref{eq:KL:discrete}, we have that $I(q \| r) = \min_{p \in \Ec} I(p \| r) = \min_{p \in \Ec \cap \Dc(r)} I(p \| r)$. It follows that we can ignore all the elements of the arrays in $\Pc_{d,n}$ that appear in the formulation of the $I$-projection problem that belong to $[n]^d \setminus \supp(r)$. After further vectorization, the $I$-projection problem can seen as consisting of attempting to $I$-project a strictly positive probability vector on a nonempty intersection of affine subspaces. This is the setting considered in Section~3 of \cite{LinSte23} and the desired result is then merely a consequence of Theorem~3.1 therein.
\end{proof}

\begin{proof}[\bf Proof of Proposition~\ref{prop:gen:h}]
  Let us verify the first claim. The function $\Lambda$ is clearly differentiable on $\R$ and thus continuous on $\R$.  Let $f$ and $g$ be the functions defined, for any $\lambda \in \R$, by
  $$
  f(\lambda) =  \sum_{\bm i \in [n]^d} h_{\bm i} \, q^\dagger_{\bm i} \exp(\lambda h_{\bm i}) \qquad \text{and} \qquad g(\lambda) = \sum_{\bm i \in [n]^d} q^\dagger_{\bm i}  \exp(\lambda h_{\bm i}).
  $$
  The derivative of the function $\Lambda$ is then
  $$
  \Lambda'(\lambda) = \frac{f'(\lambda) g(\lambda) - f(\lambda) g'(\lambda)}{g(\lambda)^2}, \qquad \lambda \in \R,
  $$
  with
  $$
f'(\lambda) = \sum_{\bm i \in [n]^d} h_{\bm i}^2 q^\dagger_{\bm i} \exp(\lambda h_{\bm i}) \qquad \text{and} \qquad g'(\lambda) = \sum_{\bm i \in [n]^d} h_{\bm i} \, q^\dagger_{\bm i} \exp(\lambda h_{\bm i}) = f(\lambda),
$$
so that
$$
\Lambda'(\lambda) = \frac{f'(\lambda) g(\lambda) - f(\lambda)^2}{g(\lambda)^2}, \qquad \lambda \in \R.
$$
Let $\lambda \in \R$  and let $z$ be the array of $\Ac_{d,n}$ defined by $z_{\bm i} = (q^\dagger_{\bm i} \exp(\lambda h_{\bm i}))^{1/2}$, $\bm i \in [n]^d$. Then, by the Cauchy-Schwarz inequality,
  $$
  f(\lambda)^2 =   \left(\sum_{\bm i \in [n]^d} h_{\bm i} z_{\bm i} \times z_{\bm i} \right)^2  < \left(\sum_{\bm i \in [n]^d} h_{\bm i}^2 z_{\bm i}^2 \right)\left(\sum_{\bm i \in [n]^d} z_{\bm i}^2 \right) = f'(\lambda) g(\lambda),
  $$
  where the strict inequality is a consequence of the fact that $(h_{\bm i} z_{\bm i})_{\bm i \in [n]^d}$ is not a scalar multiple of $(z_{\bm i})_{\bm i \in [n]^d}$  (since the array $h$ is not constant on $\supp(q^\dagger) = \supp(z)$). It follows that $\Lambda'(\lambda) > 0$ for all $\lambda \in \R$.

Let us now prove the second claim. The function $\Lambda$ is a strictly increasing bijection from $\R$ to $\mathrm{ran}(\Lambda)$. Since $a \in \mathrm{ran}(\Lambda)$, the real $\Lambda^{-1}(a)$ is well-defined and so is the array $q^\star$ given by~\eqref{eq:Iproj:gen}. It is in addition easy to verify that the latter is a probability array. Notice also from~\eqref{eq:Iproj:gen} that $\supp(q^\star) = \supp(q^\dagger)$. Furthermore,
\begin{align}
  \label{eq:hq:alpha}
  \sum_{\bm i \in [n]^d}  h_{\bm i} q^\star_{\bm i} = \frac{\sum_{\bm i \in [n]^d} h_{\bm i} \, q^\dagger_{\bm i} \exp(\Lambda^{-1}(a) h_{\bm i})}{\sum_{\bm i \in [n]^d} q^\dagger_{\bm i}  \exp(\Lambda^{-1}(a) h_{\bm i})} = \Lambda(\Lambda^{-1}(a)) = a.
\end{align}
In other words, the probability array $q^\star$ in~\eqref{eq:Iproj:gen} belongs to $\Ec''$. Since $\supp(q^\star) = \supp(q^\dagger)$, it follows that there exists $p \in \Ec''$ such that $\supp(p) = \supp(q^\dagger)$.  The fact that the $I$-projection $q^\ddagger$ of $q^\dagger$ on $\Ec''$ exists and is unique then follows from Theorem~2.1 of \cite{Csi75} (since $\Ec''$ is convex and closed). From the remark following Theorem~2.2 in the same reference, we also have that $\supp(q^\ddagger) = \supp(q^\dagger)$. Using the latter in combination with Theorem 3.1 of \cite{Csi75}, we obtain that
 \begin{equation}
  \label{eq:Iproj:lin}
  q^\ddagger_{\bm i} = q^\dagger_{\bm i} \kappa \exp \left(\lambda h_{\bm i} \right), \qquad \bm i \in [n]^d,
\end{equation}
for some unknown constants $\kappa > 0$ and $\lambda \in \R$. Combining~\eqref{eq:Iproj:lin} with the constraint $\sum_{\bm i \in [n]^d} h_{\bm i} q_{\bm i}^\ddagger = a$, we obtain that
  \begin{equation}
    \label{eq:one}
    \sum_{\bm i \in [n]^d} h_{\bm i} q^\dagger_{\bm i} \kappa \exp \left(\lambda h_{\bm i} \right) = a
  \end{equation}
  while combining~\eqref{eq:Iproj:lin} with the constraint $\sum_{\bm i  \in [n]^d} q^\ddagger_{\bm i} = 1$ gives
  \begin{equation}
    \label{eq:two}
    \kappa = \left( \sum_{\bm i \in [n]^d} q^\dagger_{\bm i} \exp \left(\lambda h_{\bm i} \right) \right)^{-1}.
  \end{equation}
 From~\eqref{eq:one} and~\eqref{eq:two}, we immediately obtain that $\Lambda(\lambda) = a$, that is, $\lambda = \Lambda^{-1}(a)$. The fact that $q^\ddagger$ is equal to $q^\star$ in~\eqref{eq:Iproj:gen} finally follows from~\eqref{eq:Iproj:lin} and~\eqref{eq:two}.
\end{proof}

\begin{proof}[\bf Proof of Lemma~\ref{lem:alpha:ran:Lambda}]
Since there exists $p \in \Ec''$ such that $\supp(p) = \supp(q^\dagger)$ and $\Ec''$ is convex and closed, Theorem~2.1 of \cite{Csi75} implies that the $I$-projection $q^\star$ of $q^\dagger$ on $\Ec''$ exists and is unique. From the remark following Theorem~2.2 in the same reference, we have that $\supp(q^\star) = \supp(q^\dagger)$. Using Theorem~3.1 of \cite{Csi75} as in the proof of the second claim of Proposition~\ref{prop:gen:h}, we next obtain that there exists $\lambda \in \R$ such that $q^\star$ is given by
$$
q^\star_{\bm i} = \frac{q^\dagger_{\bm i} \exp ( \lambda h_{\bm i} )}{\sum_{\bm i \in [n]^d} q^\dagger_{\bm i} \exp ( \lambda h_{\bm i}) }, \qquad \bm i \in [n]^d.
$$
Observe that
$$
\sum_{\bm i \in [n]^d} q^\star_{\bm i} h_{\bm i} = \frac{\sum_{\bm i \in [n]^d} h_{\bm i} q^\dagger_{\bm i} \exp ( \lambda h_{\bm i} )}{\sum_{\bm i \in [n]^d}  q^\dagger_{\bm i} \exp ( \lambda h_{\bm i})} = \Lambda(\lambda).
$$
Since $q^\star$ belongs to $\Ec''$, we finally obtain that $\Lambda(\lambda) = a$.
\end{proof}

\begin{proof}[\bf Proof of Proposition~\ref{prop:algo:gen}]
  Since there exists $p \in \Ec$ such that $\supp(p) \subset \supp(r)$, and $\Ec$ is convex and closed, Theorem~2.1 of \cite{Csi75} implies that the $I$-projection $q$ of $r$ on $\Ec$ exists and is unique.

  Let us next check that
  \begin{equation}
    \label{eq:s:J:r:J}
    \supp(s^J) = \supp(r^{(J)}) \qquad \text{for all } J \in \Jc'.
  \end{equation}
  According to Condition~\ref{cond:equal}, there exists $p \in \Ec$ such that $\supp(p) = \supp(r)$. The latter implies that $\supp(p^{(J)}) = \supp(r^{(J)})$ for all $J \in \Jc'$. But since $p \in \Ec \subset \bigcap_{J \in \Jc'} \Fc_J$, where $\Fc_J$ is defined in~\eqref{eq:Fc:J}, $p^{(J)} = s^J$ for all $J \in \Jc'$, so that we also have that $\supp(p^{(J)}) = \supp(s^J)$ for all $J \in \Jc'$, and~\eqref{eq:s:J:r:J} is verified.

  We shall now verify by induction that
  \begin{equation}
    \label{eq:induction}
    \supp(q^{[m]}) = \supp(r) \qquad \text{for all } m\geq 0.
  \end{equation}
  We have that $\supp(q^{[0]}) = \supp(r)$. Let us assume that $\supp(q^{[m-1]}) = \supp(r)$ for some $m \geq 1$. If $q^{[m]}$ is computed from $q^{[m-1]}$ via~\eqref{eq:qm:scaling}, we have that $s^J_{\bm i_J}/q^{[m-1],(J)}_{\bm i_J} > 0$ for all $\bm i \in \supp(q^{[m-1]})$ as a consequence of the fact that $\supp(q^{[m-1]}) = \supp(r)$ and~\eqref{eq:s:J:r:J} (since $\supp(q^{[m-1],(J)}) = \supp(r^{(J)}) = \supp(s^J)$ and $\bm i \in \supp(q^{[m-1]})$ implies that $\bm i_J \in \supp(q^{[m-1],(J)})$). Hence, $\supp(q^{[m]}) = \supp(q^{[m-1]})$ from~\eqref{eq:qm:scaling} and thus $\supp(q^{[m]}) = \supp(r)$. If $q^{[m]}$ is computed from $q^{[m-1]}$ via~\eqref{eq:qm:gen}, we straightforwardly have that $\supp(q^{[m]}) = \supp(q^{[m-1]})$ and thus $\supp(q^{[m]}) = \supp(r)$. Hence, \eqref{eq:induction} is proven.

  The fact $q^{[m]}$ converges to $q$ as $m \to \infty$ will then follow from Theorem 3.2 of \cite{Csi75} once we show that, for any $m \geq 1$, $q^{[m]}$ is the $I$-projection of $q^{[m-1]}$ on $\Ec_{\change{((m-1) \bmod N) + 1}}$. If $q^{[m]}$ is computed from $q^{[m-1]}$ via~\eqref{eq:qm:scaling}, the fact that $q^{[m]}$ is the $I$-projection of $q^{[m-1]}$ on $\Ec_{\change{((m-1) \bmod N) + 1}}$ follows by applying Corollary~\ref{cor:IPFP:any:marg} with $q^\dagger = q^{[m-1]}$ and $\Fc_J = \Fc_{J_{\change{((m-1) \bmod N) + 1}}}$ (the corollary is indeed applicable since $\supp(s^J) = \supp(q^{[m-1],(J)})$ by~\eqref{eq:s:J:r:J} and~\eqref{eq:induction}). If $q^{[m]}$ is computed from $q^{[m-1]}$ via~\eqref{eq:qm:gen} (we are thus computing the $I$-projection on some $\Lc_K$ in~\eqref{eq:Lc:K}), we first apply Lemma~\ref{lem:alpha:ran:Lambda} with $q^\dagger = q^{[m-1]}$, $h = h^K$, $a = \alpha_K$ and $K = K_{\change{((m-1) \bmod N) + 1}}$ (the lemma is indeed applicable since, by Condition~\ref{cond:equal} and~\eqref{eq:induction}, there exists $p \in \Ec \subset \Lc_K$ such that $\supp(p) = \supp(r) = \supp(q^{[m-1]})$) to obtain that $a \in \mathrm{ran}(\Lambda)$ and then Proposition~\ref{prop:gen:h}~(ii) with $q^\dagger = q^{[m-1]}$, $h = h^K$, $a = \alpha_K$ and $K = K_{\change{((m-1) \bmod N) + 1}}$ (the proposition is indeed applicable since $h^K$ is assumed to be non constant on $\supp(r)$, which, by~\eqref{eq:induction}, implies that $h^K$ is non constant on $\supp(q^{[m-1]})$) to obtain that $q^{[m]}$ in~\eqref{eq:qm:gen} is the $I$-projection of $q^{[m-1]}$ on $\Ec_{\change{((m-1) \bmod N) + 1}}$.
\end{proof}

\begin{proof}[\bf Proof of Proposition~\ref{prop:max:err}] $\strut$ \\
$(i) \Rightarrow (ii)$: The result follows from Corollary~\ref{cor:algo:GIS}. \\
$(ii) \Rightarrow (iii)$: Let $A$ and $B$ be the functions from $\Pc_{d,n}$ to $[0,\infty)$ defined by
  $$
  A(p) := \max_{J \in \Jc'} \max_{\bm i_J \in [n]^{|J|}} \left| p^{(J)}_{\bm i_J} - s^J_{\bm i_J} \right| \qquad \text{and}  \qquad B(p) := \max_{K \in \Kc} \left| \sum_{\bm i \in [n]^d} p_{\bm i} h^K_{\bm i} - \alpha_K \right|.
  $$
 Since $(ii)$ holds, $q' = \lim_{m \to \infty} q^{[m]} = q$, the $I$-projection of $r$ on $\Ec$ which satisfies $\supp(q) \subset \supp(r)$ by definition (see Section~\ref{sec:Iproj}). Since $q \in \Ec$ and the functions $A$ and $B$ are continuous, $\lim_{m \to \infty} \Err_{\Jc'}^{[m]} = \lim_{m \to \infty} A(q^{[m]}) = A(q) = 0$ and $\lim_{m \to \infty} \Err_{\Kc}^{[m]}  = \lim_{m \to \infty} B(q^{[m]}) = B(q) = 0$.\\
$(iii) \Rightarrow (i)$: Since $q' = \lim_{m \to \infty} q^{[m]}$ exists, $\lim_{m \to \infty} \Err_{\Jc'}^{[m]} = 0$ implies that $A(q') = 0$ and $\lim_{m \to \infty} \Err_{\Kc}^{[m]}  = 0$ implies that $B(q') = 0$. Hence, $q' \in \Ec$. This, combined with the fact that $\supp(q') \subset \supp(r)$, implies that Condition~\ref{cond:included} holds.
\end{proof}
    
\section*{Acknowledgments}

The authors would like to thank an Associate Editor and two anonymous Referees for their constructive comments on an earlier version of this manuscript.

\bibliographystyle{myjmva}
\bibliography{biblio}

\newpage

\includepdf[page={-},offset=0mm 0mm]{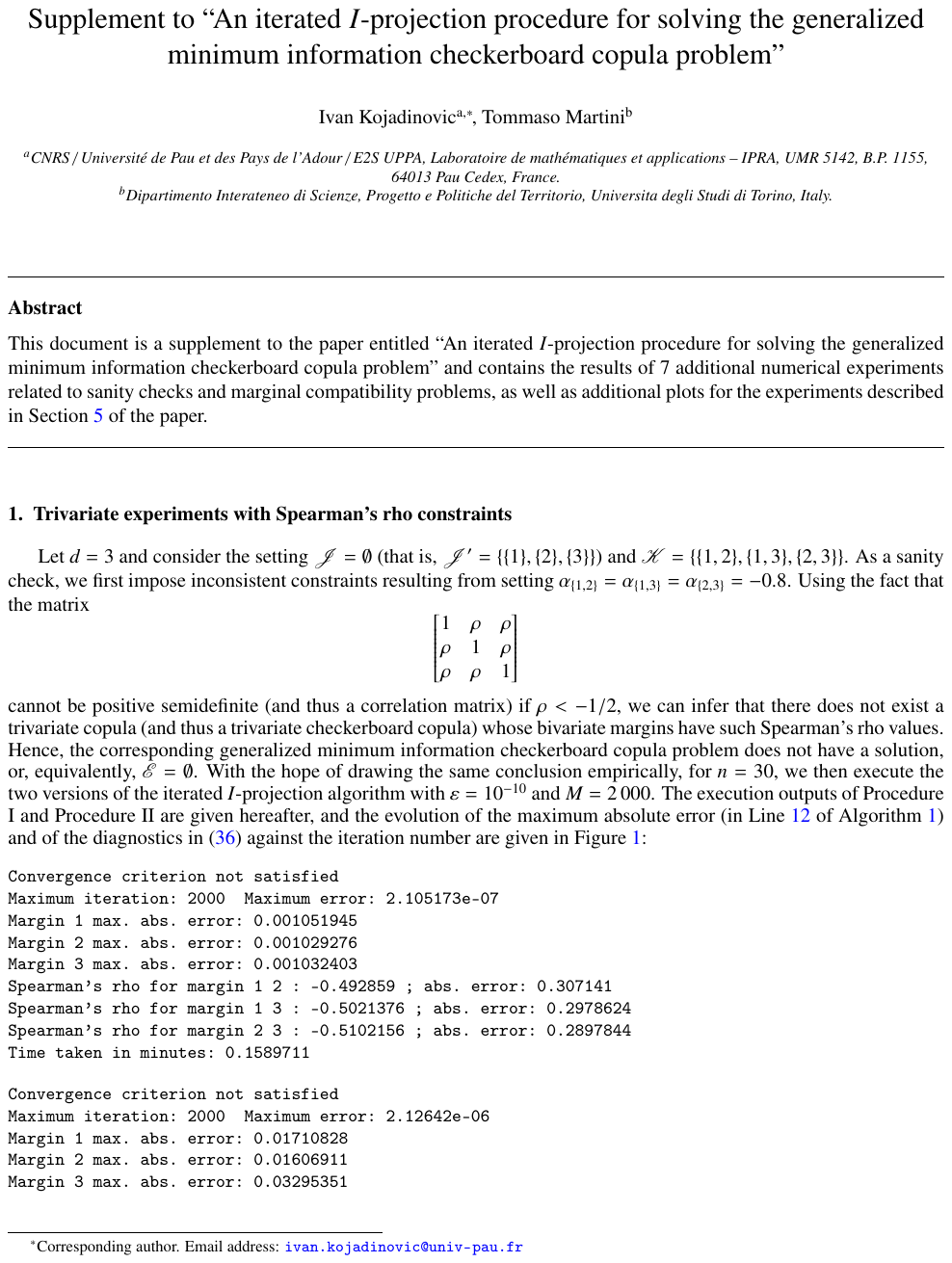}

\end{document}